\documentclass[final]{amsart}
\usepackage{amsfonts, amssymb, amsmath}
\usepackage[pdftex]{hyperref}

\usepackage{multirow}
\usepackage{showlabels}

\newtheorem{thm}[subsection]{Theorem}
\newtheorem{lem}[subsection]{Lemma}
\newtheorem{cor}[subsection]{Corollary}
\newtheorem{rem}[subsection]{Remark}
\newtheorem{prop}[subsection]{Proposition}
\newtheorem{conj}[subsection]{Conjecture}
\newtheorem{hypo}{Hypothesis}

\theoremstyle{definition}
\newtheorem{defi}[subsection]{Definition}
\newtheorem{ex}[subsection]{Example}
\newtheorem{notation}[subsection]{Notation}
\newtheorem{setting}[subsection]{Notation}
\newtheorem{generalass}[subsection]{General assumption}
\newtheorem{ass}[subsection]{Assumption}
\newtheorem{abschnitt}[subsection]{}

\newcommand{\Syml}{\Sym_{\pm \overline l}}
\newcommand{\je}{{j'}}

\newcommand{{\cS}}{{\mathcal S}}
\newcommand{{\cK}}{{\mathcal K}}
\newcommand{{\cN}}{{\mathcal N}}
\newcommand{{\cW}}{{\mathcal W}}

\newcommand{\calZ}{\mathcal Z}
\newcommand{\calQ}{\mathcal Q}

\newcommand{\ovf}{\overline f}
\newcommand{\ovphi}{\overline \phi}
\newcommand{\sgn}{\ensuremath{\operatorname {sgn}}}
\newcommand{\ophi}{\ensuremath{\overline \phi}}

\newcommand{\ovovT}{\ensuremath{\overline \ovT}}
\newcommand{\ovW}{\overline W}
\newcommand{\ovY}{\overline Y}
\newcommand{\ovovG}{\overline \ovG}
\newcommand{\ovovN}{\overline \ovN}
\newcommand{\ovovX}{\overline \ovX}
\newcommand{\osigma}{\ensuremath{\overline \sigma}}
\newcommand{\ovrho}{\ensuremath{\overline \rho}}
\newcommand{\oF}{\ensuremath{\overline F} { }}

\newcommand{\la}{\ensuremath{\lambda}}

\newcommand{\wpsi}{\ensuremath{\widetilde \psi}}

\newcommand{\IRRl}{\operatorname{Irr}_{\ell'}}
\newcommand{\wN }{\ensuremath{\widehat N}}
\newcommand{\wK }{\ensuremath{\widehat K}}
\newcommand{\wT }{\ensuremath{\widehat T}}

\newcommand{\calN}{\ensuremath{\mathcal N }}
\newcommand{\calI}{\ensuremath{\mathcal I }}
\newcommand{\cI}{\ensuremath{\mathcal I }}

\newcommand{\PHI}{{\Phi}}
 
\newcommand{\tw}[1]{\,{}^#1\!}
\newcommand{\ovG}{\mathbf G}
\newcommand{\ovX}{\mathbf X}
\newcommand{\ovT}{\mathbf T}

\newcommand{\ovS}{{\mathbf S}}
\newcommand{\ovN}{\mathbf N}
\newcommand{\ovK}{\mathbf K}

\newcommand{\GG}{\ensuremath{\mathbb{G}}}
\newcommand{\TT}{\ensuremath{\mathbb{T}}}
\newcommand{\SSS}{\ensuremath{\mathbb{S}}}

\newcommand{\ovR}{\ensuremath{\overline R}{ }}
\newcommand{\Hom}{\operatorname{Hom}}

\newcommand{\ZZ}{\ensuremath{\mathbb{Z}}}
\newcommand{\CC}{\ensuremath{\mathbb{C}}}

\newcommand{\bw}{{\mathbf w}}
\newcommand{\bww}{{\widetilde {\mathbf w_{0}}}}
\newcommand{\bbw}{{ {\mathbf w_{0}}}}
\newcommand{\bB}{{\mathbf B}}
\newcommand{\B}{\ensuremath{{\mathbf B}}{ }}

\newcommand{\Zent}{\ensuremath{{\rm{Z}}}}
\newcommand{\Cent}{\ensuremath{{\rm{C}}}} 
\newcommand{\Stab}{{\mathrm{Stab}}}
\newcommand{\NNN}{\ensuremath{{\rm{N}}}} 
\newcommand{\Cy}{\Cent}

\newcommand{\I}{\ensuremath{\operatorname{I}}}

\newcommand{\Irro}{\ensuremath{\operatorname{Irr}}}
\def\Irr(#1){\ensuremath{\Irro\left( #1 \right)}}
\def\Irrl(#1){\ensuremath{\Irro_{\ell'}\left( #1 \right)}}
\def\restr#1|#2{\left.#1\right\rceil_{#2}}

\newcommand{\bb}{\ensuremath{\mathrm b}}

\newcommand{\FF}{\ensuremath{\mathbb{F}}} 
\newcommand{\ovF}{\overline \FF}

\newcommand{\calB}{{\mathcal B}}

\newcommand{\calU}{{\mathcal U}}
\newcommand{\calL}{{\mathcal L}}
\newcommand{\calK}{{\mathcal K}}
\newcommand{\Lang}{{\mathfrak {L}}}

\newcommand{\SC}{sc}
\newcommand{\Alt}{\ensuremath{\mathfrak A}} 

\newcommand{\SL}{{\operatorname{SL}}}
\newcommand{\SU}{{\operatorname{SU}}}

\newcommand{\Sp}{{\operatorname{Sp}}}
\newcommand{\tA}{\mathsf A}
\newcommand{\tB}{\mathsf B}
\newcommand{\tC}{\mathsf C}
\newcommand{\tD}{\mathsf D}
\newcommand{\RF}{{\ensuremath {R_F}}}

\newcommand{\calG}{\ensuremath{\Gamma}}

\newcommand{\al}{{\alpha}}

\makeatletter
\def\IRR(#1){\IRR@h#1@}	
\def\IRRL(#1){\IRRL@h#1@}	
\def\IRR@h#1|#2@{\Irro \left(\left.#1\vphantom{#2}\hskip.1em\,\right|\,\relax #2\right)}
\def\IRRL@h#1|#2@{\Irro_{\ell'} \left(\left.#1\vphantom{#2}\hskip.1em\,\right|\,\relax #2\right)}
\def\nicefrac#1#2{#1/#2}

\def\Spann<#1>{\Spann@h#1@}
\def\Spann@h#1|#2@{\left\langle\left.#1\vphantom{#2}\hskip.1em\right|\,\relax #2 \right\rangle}
\def\Set#1{\Set@h#1@}
\def\Lset#1{\Lset@h#1@}
\def\Set@h#1|#2@{\left\{\left.#1\vphantom{#2}\hskip.1em\,\right|\,\relax #2\right\}}
\def\Bset@h#1in#2|#3@{\Set{{#1\qin#2}|{#3}}}
\def\Lset@h#1@{\left\{#1\right\}}
\makeatother
\newcommand{\Sym}{\mathfrak S}
\newcommand{\Inn}{\operatorname{Inn}}
\newcommand{\Aut}{\operatorname{Aut}}

\newcommand{\betrag}[1]{\left|{#1}\right|}
\def\spann<#1>{\left\langle#1\right\rangle}
\def\floor#1{\lfloor #1 \rfloor}
\newcommand{\teilt}{\mathop{\rule[-1ex]{.5pt}{3ex}}}
\def\nteilt{{\nmid}}
\begin{document}

\title{Regular Sylow $d$-Tori of classical groups and the McKay conjecture}

\author{Britta Sp\"ath}
\thanks{This research has been supported by the DFG-grant ``Die Alperin-McKay-Vermutung f\"ur endliche Gruppen''.}
\address{ Institut de math\'ematiques de Jussieu, Universit\'e Paris VII, 175 rue du Chevaleret, 75013 Paris, France }
\email{spath@math.jussieu.fr}
\date{}

\begin{abstract}We prove for finite reductive groups $G$ of classical type, that every character $\chi \in \Irr(L)$ extends to its inertia group in $N$, where $L$ is an abelian centraliser of a Sylow $d$-torus $\ovS$ of $G$ and $N:=\NNN_G(\ovS)$. This gives a precise description of the irreducible characters of $N$ and proves $\betrag{\Irrl(G)}=\betrag{\Irrl(N)}$ according to \cite{Ma06} for all primes $\ell$ determined by $\ovS$. Furthermore this enables us to verify the McKay conjecture in this situation for some primes. \end{abstract}

\maketitle

\section{Introduction}
For any finite group $H$ and prime $\ell$ the set $\Irrl(H)$ denotes all ordinary irreducible characters of $H$ whose degree is prime to $\ell$. McKay observed for some finite simple groups and $\ell=2$ that the equation $\betrag{\Irrl(H)}= \betrag{\Irrl(\NNN_H(P))}$ holds, where $\NNN_H(P)$ is the normaliser of a Sylow $\ell$-subgroup $P$ of $H$. In generalisation of this Alperin formulated the now called McKay conjecture, namely that for every group $H$ and every prime $\ell$ the equation $\betrag{\Irrl(H)}= \betrag{\Irrl(\NNN_H(P))}$ holds.

This is meanwhile proved for several families of groups, e.g., Lehrer has verified it in \cite{LehrerMcKay} for groups of Lie type with conformal type over $\FF_q$ and also for $H= \SL_n(q)$, whenever $\ell$ divides $q$. Michler and Olsson proved in \cite{MiOl} that the Alperin-McKay conjecture, a more refined version of the McKay conjecture, holds for the general linear groups and the unitary groups over finite fields $\FF_q$ and primes $\ell$ with $\ell\nmid q$. 

The results of Isaacs, Malle and Navarro in \cite{IsaMaNa} give a new approach to a general proof for all finite groups. They show that a finite group $H$ fulfils the McKay conjecture for a prime $\ell$ if every simple non-abelian section of $H$, whose order is divisible by $\ell$ is {\it good for $\ell$}. They call a simple non-abelian group $X$ good for $\ell$ if the maximal central perfect $\ell'$-extension $G$ of $X$ satisfies a set of conditions. Amongst them they claim the existence of a subgroup $N\geq \NNN_G(P)$ with $N\neq G$, $\Inn(G) \Stab_{\Aut(G)}(N)= \Aut(G)$ and a bijection between $\Irrl(G)$ and $\Irrl(N)$, where $P$ is a Sylow $\ell$-subgroup of $G$. Furthermore this bijection has to satisfy certain equivariance-conditions, namely it has to be $\Stab_{\Aut(G)}(N)$-equivariant and preserve cohomological elements associated to the characters.

In \cite{ManonLie} these conditions have already been proved for all simple groups not of Lie type. The aim of this paper is twofold: on the one hand we prove that the irreducible characters of the centraliser $\Cent_G(\ovS)$ of $\ovS$ in $G$ extend to their inertia group in the normaliser $N:=\NNN_G(\ovS)$ of $\ovS$, where $G$ is a classical group, $\ovS$ a Sylow $d$-torus of $G$ for some integer $d$ and $\Cent_G(\ovS)$ is abelian. By \cite{Ma06} this is relevant for verifying the inductive McKay condition for $\nicefrac G {\Zent(G)}$ and primes different from the defining characteristic $\ell$. More precisely, Malle constructed for all primes $\ell$ with $\ell\teilt \betrag{G}$ different from the defining characteristic a group $\NNN_G(P)\leq N < G$ and a bijection $':\Irrl(G)\rightarrow \Irrl(N)$ using a generalisation of Theorem \ref{Theo1}. It could already be shown that this bijection fulfils some of the desired properties. On the other hand we deduce from this result that the McKay conjecture holds for these groups and some primes.

We assume that the centraliser of the Sylow $d$-torus is abelian, and call Sylow $d$-tori with this property {\it regular} and numbers $d$, for which the Sylow $d$-torus of $(\ovG,F)$ is regular, {\it regular numbers of $(\ovG,F)$}, as well. A subsequent paper will generalise the result to arbitrary Sylow $d$-tori and uses the results of this case beside different methods, which are necessary as the desired extensions are non-linear characters. The restriction to certain tori corresponds to an assumption on the prime $\ell$, for which the McKay conjecture will be verified.

The key aim of the paper is to establish the following theorem, which is an analogue of \cite[Theorem A]{Spaeth_Preprint} for classical groups with an additional restriction on the Sylow $d$-torus. 

\begin{thm}\label{Theo1}
Let $\ovG$ be a simply-connected simple algebraic group defined over $\FF_{q}$ via the Frobenius endomorphism $F:\ovG\rightarrow \ovG$, where the root system of $\ovG$ is of classical type and such that $\ovG^F$ is not Steinberg's triality group. Furthermore let $\ovS$ be a Sylow $d$-torus of $(\ovG,F)$ for some integer $d$, such that $L:=\Cent_{\ovG^F}(\ovS)$ is abelian. Then every irreducible character of $L$ extends to its inertia group in $N:=\NNN_{\ovG^F}(\ovS)$.
\end{thm}

This statement is part of the author's PhD thesis \cite{Spaeth} under the supervision of Gunter Malle.
The case where $\ovG^F= \tw{3}\tD_4(q)$ has already been dealt with in \cite{Spaeth_Preprint}. 

\begin{cor}\label{Cor1}
 Let $\ovG$ be a simply-connected simple algebraic group defined over $\FF_{q}$ via the Frobenius endomorphism $F:\ovG\rightarrow \ovG$, where the root system of $\ovG$ is of classical type and such that $\ovG^F$ is not Steinberg's triality group. Let $\ell>3$ be a prime with $\ell \nteilt q$, such that the order $d$ of $q$ in $(\nicefrac \ZZ {\ell \ZZ})^\star$ is a regular number of $(\ovG,F)$. Then the McKay conjecture holds for $\ovG^F$ and $\ell$.
\end{cor}

The proof also shows the statement for $\ell \in \Lset{2,3}$ with $\ell \nmid q$ with some exceptions given in Remark \ref{Ausnahmen}. We hope that with more knowledge about the extensions from Theorem \ref{Theo1}, one could construct a bijection with the properties desired in \cite{IsaMaNa} and prove hereby that the associated simple group is good for $\ell$.

The paper is structured in the following way: in Section \ref{sec_not} we recall some essential facts about finite reductive groups and Sylow tori, mainly known from \cite{BrMa}. In the following section we introduce Sylow twists and how the groups occurring in Theorem \ref{Theo1} can be obtained using these elements. The succeeding section is concerned with tools from character theory.

For proving Theorem \ref{Theo1} we first prove Proposition \ref{mFreg}, by which maximal extensibility holds for a pair of groups, if they are closely related to wreath products and several assumptions are satisfied for them. In the following section we present subgroups of $\ovG$, which satisfy some of the wanted assumptions. Afterwards we prove Theorem \ref{Theo1} for the different classical groups $G$ separately in Sections \ref{applC}-\ref{appl2D}. While we apply Proposition \ref{mFreg} very directly and obtain Theorem \ref{Theo1} in the case where $\ovG$ has a root system of type $\tC_l$ in Section \ref{applC}, in Sections \ref{applA}-\ref{applB} more considerations are made for dealing with the cases, where $\ovG$ has a root system of type $\tA_l$ or $\tB_l$. In contrast to these cases, an embedding of a group of type $\tD_{l}$ into one of type $\tB_{l}$ proves the remaining cases in the sections \ref{applD} and \ref{appl2D}. Finally we show how one can deduce Corollary \ref{Cor1} from Theorem \ref{Theo1}.
\smallskip

\noindent {\bf Acknowledgement}
The author likes to thank the Mathematical Science Research Institute for its hospitality and financial support. In addition the author wants to thank Gabriel Navarro for pointing out the proof of Lemma \ref{lem_Navarro} and Gunter Malle for introducing her to this interesting topic. Furthermore the author thanks Marc Cabanes and Michael Ehrig for careful reading an earlier version of this paper and many helpful comments.

\section{Notation and basic definitions} \label{sec_not}
In this section we introduce the general framework for our further calculations. We establish the notation used in the whole paper and recapture some basic facts about generic groups introduced by Malle and Brou\'e in \cite{BrMa}.

\begin{notation}\label{basicsetting} Let $p$ be a prime number and $\ovG$ the simply-connected simple algebraic group with irreducible root system $R$ over the algebraic closure $\ovF_{p}$ of the field $\FF_{p}$ with $p$ elements. We choose a system $\RF$ of simple roots in $R$. According to \cite[9.4]{Springer} the group $\ovG$ has a presentation with the generators $x_\al(t)$ ($t\in \ovF_p$, $\al \in R$). We use the notation of \cite[12.1]{Car1}, in particular the definition of the elements $h_{\al}(t')$, $n_{\al}(t')$ and $x_{\al}(t)$ ($t\in \ovF_{p}$, $t'\in \ovF_{p}^{*}$, $\al \in R$, where $\ovF_{p}^{*}$ denotes the multiplicative group of $\ovF_{p}$). We denote by $\ovX_\al$ ($\al \in R$) the group $\Spann<x_\al(t)|t\in \ovF_p>$, the root subgroup associated to $\al$.

In this situation $\ovT:=\Spann<h_\al(t)| \al \in R, t \in \ovF_p^*>$ is a maximal torus of $\ovG$ with normaliser $\ovN:=\NNN_\ovG(\ovT)= {\Spann<n_\al(t)| \al \in R, t \in \ovF_p^*>}$. The quotient group $\nicefrac {\ovN}{\ovT}$ is the Weyl group $W$ of $\ovG$, hence isomorphic to the reflection group of $R$. We denote the corresponding epimorphism by $\rho:\ovN \rightarrow W$ with $n_\al(t)\mapsto w_\al$ for $\al \in R$, $t \in \ovF_p^*$, where $w_\al$ is the reflection along the root $\alpha \in R$. 

\label{EinfV}
The extended Weyl group ${V:=\Spann<n_\al(1)| \al \in R>}$ was introduced in \cite{Tits} by Tits. Due to the relation $ n_\al(t)= h_\al(-t) n_\al(1)$ this group fulfils $ \ovN= \spann<V,\ovT>$ and is generated by the elements $n_i:=n_{\al_i}(1)$, where $\RF=\Lset{\al_1,\cdots,\al_r}$. The Steinberg relations imply that ${H:= \Spann<h_\al(-1)| \al \in R>}$ satisfies $H=V\cap \ovT$ and has order $(2,q-1)^{\betrag{\RF}}$. (We denote the greatest common divisor of $a,b\in \ZZ$ ($a,b>0$) by $(a,b)$.) The isomorphism type of $V$ is independent of $q$, whenever $2\nmid q$.
\end{notation}

A finite reductive group is the subgroup of fixed points of a Frobenius endomorphism. We mainly use the following ones.

\begin{setting}\label{setgen} 
Let $q$ be a power of $p$. The endomorphism $F_0:\ovG\rightarrow \ovG$ is defined by $ x_\al(t)\mapsto x_\al(t^q)$ ($\al \in R$, $t \in \ovF_p$) and called {\it the standard Frobenius endomorphism} associated to $q$.
 If the Dynkin diagram associated to $\RF$ has a length preserving symmetry ${\gamma:\RF\rightarrow \RF}$, the associated automorphism $\Gamma_0:\ovG\rightarrow \ovG$ acts via \[ x_\al(t)\mapsto x_{\gamma(\al)}(t) \,\,\text{ ($\pm \al \in \RF$, $t \in \ovF_p$) }.\] The composition $\Gamma_0\circ F_0 $ is also a Frobenius endomorphism on $\ovG$ in the sense of \cite[1.17]{Car2}.
Let $\Gamma\in \spann<\Gamma_0>$ and $F=F_0\circ \Gamma$. The maximal torus $\ovT$ of $\ovG$ is $F$-stable.

\item The triple $(\ovG, \ovT, F)$ defines a quintuple $\GG:=(X, R,Y , R^\vee, W\phi)$ called the {\it generic group}, where $X:=\Hom(\ovT,\ovF_p^{*})$, $Y:=\Hom(\ovF_p^{*},\ovT)$, $R^{\vee}$ are the coroots of $\ovG$, $\phi$ is the automorphism of finite order of $Y$ and $Y\otimes \CC$, respectively, induced by $F$, and $W\phi$ its coset in the automorphisms of $Y$. 
More concretely $\phi$ is chosen such that the action of $F$ on $\ovT$ induces the morphism $q\phi$ on $Y$.
In addition there exists a duality $\langle -,- \rangle:X\times Y \rightarrow \ZZ $ with $\langle \al, \al^\vee\rangle = 2$ (see \cite[Section 2]{BrMa}).

While the root datum $(X,R,Y,R^\vee)$ determines the reductive group over $\ovF_p$ up to isomorphism, the group $\ovG^F$ is similarly determined by the corresponding generic group $\GG$ and the prime power $q$. Thus we may also denote $\ovG^F$ by $\GG(q)$. The polynomial order $\betrag{\GG}$ of $\GG$ is a monic polynomial of the form \[ \betrag{\GG}(x)= x^N\prod_{d\in \ZZ_{>0}} \PHI_d^{a(d)}(x)\in \ZZ[x],\] 
 where $N\in \ZZ_{\geq 0 }$, $\PHI_d$ is $d$th cyclotomic polynomial and $a(d) \in \ZZ_{\geq 0}$ with the property, that $\betrag{\GG(q')}= \betrag{\GG}(q')$ holds for every prime power $q'$,
\end{setting}

We make frequent use of the following connection between simply-connected simple groups associated to different root systems.
\begin{rem}[{\cite[2.1.9]{Spaeth}}]\label{einbett} 
Let $R'$ be an irreducible root subsystem of $R$ such that $R'$ is parabolic or $\ZZ R^\vee=\ZZ {R'}^\vee$. Then $\ovG_{R'}:=\Spann<\ovX_\al| \al \in R'>$ is isomorphic - as a group - to a simply-connected simple algebraic group $\ovG'$ associated to $R'$ over the same field.\end{rem}

\begin{proof} Obviously $\ovG_{R'}$ satisfies the Steinberg relations. Hence there exists a morphism $\ovG'\rightarrow \ovG_{R'}$. According to \cite[Corollary 1, p. 39]{Steinb} its kernel lies inside the centre of $\ovG'$. Hence it suffices to show that no central element vanishes via this morphism. If $R'$ is a parabolic root system with root subsystem $\RF'$, the torus $\ovT$ can be embedded into $\ovG$ because of \cite[12.1]{Car1}. As every central element lies in $\ovT$, it doesn't vanish via the morphism. Similarly this can also be shown in the second case. \end{proof}

We further mention some constructions of generic groups which we used. In \cite{BrMa} some substructures of generic groups are defined. They correspond to $\ovG^F$-conjugacy classes of $F$-stable subgroups of $\ovG^F$. 

\begin{rem}[Tori and Levi subgroups of generic groups]
A generic group of the form $(\nicefrac {X}{Y'^\perp},Y', \restr w\phi|{Y'})=(\nicefrac {X}{Y'^\perp},\emptyset,Y', \emptyset,\restr w\phi|{Y'})$ with $w\in W$ and a $w\phi$-stable sublattice $Y'$ of $Y$ is called {\it torus of $\GG$}. (The sublattice $Y'^\perp$ of $X$ is defined by the bilinear form on $X\times Y$ in the natural way.) The polynomial order of a generic torus $\TT=(X,Y,\phi)$ coincides with the characteristic polynomial of $\phi$ on $Y\otimes \CC$, i.e., $\betrag{\TT}(x)= \det_{Y\otimes \CC}(x\phi-1)$.

A {\it Levi subgroups of $\GG$} is a generic groups of the form $(X,R',Y,{R'}^{\vee}, W_{R'}w\phi)$ for $w\phi \in W\phi$ and a $w\phi$-stable parabolic root subsystem $R'$ of $R$. (By $W_{R'}$ we denote the subgroup of $W$ generated by the reflections along the roots of $R'$.) 

 Further one can associate to a generic torus $\SSS = (X',Y', \restr w\phi|{Y'})$ ($w\in W$) of $\GG$ its centraliser in $\GG$. This is the Levi subgroup $ \Cent_\GG(\SSS)= (X,R',Y,R'^\vee, W_{R'}w\phi)$ with $R':=R\cap Y'^{\perp}$.

 The generic groups of $F$-stable tori and $F$-stable Levi subgroups of $\ovG$ are generic tori and generic Levi subgroups of $\GG$, respectively. Also taking the centraliser of a torus and computing the generic group commute with each other. 

 The {\it Sylow $d$-tori} of $\GG$ are defined to be the tori $\SSS$ of $\GG$ with $\betrag{\SSS}=\PHI_d^{a(d)}$. An $F$-stable torus of $\ovG$ whose generic group is a Sylow $d$-torus, is also called a Sylow $d$-torus of $(\ovG,F)$ or $\ovG^F$. The existence and conjugacy in $\ovG^{F}$ of all Sylow $d$-tori of $(\ovG, F)$ was proven in \cite[3.4]{BrMa}.
\end{rem}

In our further considerations we will mainly concentrate on groups associated to Sylow tori. 

\begin{defi} 
Let $\ovS$ be a Sylow $d$-torus of $(\ovG,F)$. We call $\Cent_{\ovG^F}(\ovS)$ {\it Sylow $d$-Levi subgroup} and $\NNN_{\ovG^F}(\ovS)$ the associated {\it Sylow $d$-normaliser}. By abuse of notation we call $\ovS$ Sylow torus, if $\ovS$ is a Sylow $d$-torus for some $d$. Sylow Levi subgroup and Sylow normaliser are analogously defined.

We call a Sylow torus $\ovS$ {\it regular}, if its centraliser in $\ovG$ is a torus. 
An integer $d$ is called {\it regular number of $(\ovG,F)$}, if there exists a regular Sylow $d$-torus of $(\ovG,F)$.
\end{defi} 

These numbers can be determined using the theory of regular elements from \cite{Springer74}. Thereby we use the following notion, which is adapted to our situation.

\begin{defi}
Let $w\phi\in W\phi$. A vector $ y \in Y\otimes \CC$ is called {\it regular}, if it doesn't lie in any reflecting hyperplane of $W$. If there exists a regular vector, which is at the time a eigenvector of $w\phi$ to the eigenvalue $\zeta$ the element is called {\it $\zeta$-regular element or regular element of $W\phi$} . An integer $d$ is called a {\it regular number of $W\phi$}, if there exists a $\zeta$-regular element, where $\zeta$ is a primitive $d$th root of unity.
\end{defi}

In our later considerations we may assume $d$ to be a regular number of $W\phi$ because of the first part of the following lemma. In several proofs we avoid dealing with Sylow $d$-Levi subgroups and Sylow $d$-normalisers of $(\ovG,F)$ and prefer analysing isomorphic groups, which are obtained by using a different Frobenius endomorphism.

\begin{rem}[{\cite[2.5, 2.6]{Spaeth_Preprint}}]\label{rem_regular}
\begin{enumerate}
 \item The integer $d$ is a regular number of $(\ovG,F)$ if and only if $d$ is a regular number of $W\phi$.

 \item \label{conj} Let $L$ be a Sylow $d$-Levi subgroup of $(\ovG, F)$ and $L'$ a Sylow $d$-Levi subgroup of $(\ovG, gF)$ for some $g\in \ovG$. The groups $L$ and $L'$ are conjugate to each other in $\ovG$.
\end{enumerate}
\end{rem}

For our proof of Theorem \ref{Theo1} we need the regular numbers of $(\ovG,F)$ and have gathered them in the table below. By the above lemma they coincide with the regular numbers of $W\phi$, which are known from \cite[Section 5]{Springer74}.
\begin{table}[ht]
\begin{center}
\begin{tabular}{|l|l|l|l|l|l|l|}
\hline
Type of&$\,\,\tA_{l-1}$& $^2\tA_{l-1}$ ($2\teilt l$) &$^2\tA_{l-1}$ ($2\nteilt l$)& $\tB_{l}$, $\tC_{l}$& $\,\,\tD_{l}$ & $^2\tD_l$\\ $(\ovG,F)$&&&&&&\\ \hline
Regular&
$d\teilt l$& $d\teilt l$& $d\teilt (l-1)$ & & $d\teilt 2l-2$& $d\teilt 2l-2$\\
 numbers& or& or& or& $d\teilt 2l$& or& or\\
of $(\ovG,F)$ & {$d\teilt (l-1)$}& {$2\teilt d$ and $d\teilt 2l-2$}& {$2\teilt d$ and $d\teilt 2l$}& & {$d\teilt l$}& {$d\teilt 2l$ and $d\nteilt l$}\\
\hline
\end{tabular}
\end{center}
\caption{Regular numbers for classical groups} \label{TabregZahl}
\end{table}

\section{Construction of Sylow Levi groups and Sylow normalisers}\label{sec_tools}\label{constructions} 
For the proof of Theorem \ref{Theo1} we need a good way to construct Sylow $d$-Levi subgroups and Sylow $d$-normalisers. This will be done with the help of Sylow $d$-twists, which were already introduced in \cite[Section 3]{Spaeth_Preprint}.

\begin{defi}[Sylow $d$-twist]\label{DefSylowtwist}
Let $d$ be a positive integer and $\Gamma$ the automorphism of $\ovG$ from \ref{setgen}. Further let $\phi_d$ and $a(d)$ be defined as in \ref{setgen}. An automorphism of the form $v\Gamma$ ($v\in \ovN$) is called a {\it Sylow $d$-twist of $(\ovG,F)$}, if $\PHI_d^{a(d)} \teilt \det_{ Y\otimes \CC}(x \rho(v)\phi-1)$, i.e., $\PHI_{d}^{a(d)}$ is a divisor of the characteristic polynomial of $\rho(v)\phi$ on $Y\otimes \CC$.
We call $v\Gamma$ also a Sylow twist, if it is a Sylow $d$-twist for some $d$.
\end{defi}

In the case where $d$ is a regular number of $(\ovG,F)$, Sylow twists can be constructed with morphisms between the braid group associated to $R$, the extended Weyl group $V$ already mentioned in \ref{EinfV} and the Weyl group $W$. Thereby one uses also the good $d$th roots of $\bbw^2$, which were introduced in \cite{BrMi}. We recapture shortly their definition.

\begin{defi}\label{def_good_roots} \label{defgoodroot}
Let $\bB$ be the braid group associated to $R$ with generators $ \bb_1, \ldots ,\break  \bb_{\betrag{\RF}}$, $\phi_{\bB}$ the automorphism of $\bB$ associated to $\phi$, $\bbw\in \bB$ the element corresponding to the longest element $w_0$ in $W$ and $\bB_{\rm{red}}^{+}$ the set of reduced elements in $\bB$ with positive length.
Then we call $\bw$ a {\it good $d$th root of $\bbw^2$}, if $\bw^d=\bbw$ and for $i\leq \frac d 2 $ the element $\bbw^i$ is reduced. If $(\bww \phi_{\bB})^d = \bbw^2\phi_{\bB}^d$ and $(\bww\phi_{\bB})^i\in \bB_{\rm{red}}^+\phi_{\bB}^i$ is reduced for every $1\leq i\leq \frac d 2$, then $\bw\phi_{\bB}$ is called a {\it good $d$th $\phi_\bB$-root of $\bbw^2$}.

Let $\tau: \bB\rightarrow V$ be the epimorphism with $ \bb_i\mapsto n_{i} $ for all $1\leq i \leq \betrag{\RF}$, where $n_i$ are the generators of $V$ introduced in \ref{basicsetting}, which satisfy the braid relations according to \cite[9.3.2]{Springer}. We call the element $\tau(\bw_0)= \bww$ the {\it canonical representative of $w_0$ in $\ovG$}.
\end{defi}

\begin{rem}[{\cite[3.2]{Spaeth_Preprint}, \cite[Theorem 1.3]{Bessis}}]\label{remsylow} 

\begin{enumerate}
 \item The automorphism $v\Gamma$ is a Sylow $d$-twist of $(\ovG,F)$ if and only if $\rho(v)\phi$ is a $\zeta$-regular element of $W\phi$ for some $\zeta$, which is a primitive $d$th root of unity.
 \item \label{remsylowb}
Let $\bw\phi$ be a good $d$th $\phi_{\bB}$-root of $\bbw$. Then $\tau(\bw)\Gamma$ is a Sylow twist of $(\ovG, F)$.
 \item \label{bembessis} If $R$ is of classical type, $\Gamma=\operatorname{id}_{\ovG}$ and $\bw$ is as in (b), $\tau(\bw)$ is a Sylow $d$-twist of $(\ovG,F)$ fulfilling $\rho(\Cent_V(v))= \Cent_W(\rho(v))$.
\end{enumerate}
\end{rem} 
\begin{proof} The parts (a) and (b) have been dealt in Remark 3.2 of \cite{Spaeth_Preprint}. The proof of (a) relies on the following fundamental property of regular elements: for every $\zeta$-regular element $w\phi$ of $W\phi$ the polynomial $\Phi_d^{a(d)}$ divides the characteristic polynomial of $w\phi$ on $Y\otimes \CC$. The part (b) relies on the fact that $\rho(\tau(\bw))\phi$ is a $\zeta$-regular element by \cite[3.14, 6.6]{BrMi}. The remaining part follows from Theorem 1.3 in \cite{Bessis}. A detailed proof, how the theorem implies the statement, can be found in \cite[2.3.14]{Spaeth}.
\end{proof}

A Sylow $d$-twist $v\Gamma$ determines a generic Sylow $d$-torus of $\GG$ and the corresponding algebraic torus of $(\ovG,vF)$, where $vF$ acts via $x\mapsto F(x^v)$.

\begin{lem}[Construction of Sylow $d$-tori, {\cite[3.3]{Spaeth_Preprint}}] \label{K:S}
 Let $d$ be an integer and $v\in \ovN$ such that $v\Gamma\in \ovN\Gamma$ is a Sylow $d$-twist.
 Further let ${Y':=Y\cap \ker_{Y\otimes \CC}(\PHI_d(\rho(v)\phi))}$, i.e., $Y'\otimes \CC $ is the sum of the eigenspaces of $\rho(v)\phi$ corresponding to primitive $d$th roots of unity, and $X':=\nicefrac{X}{Y'^\perp}$.
 Then the generic torus $\SSS:=(X',Y',\rho(v)\phi)$ is a Sylow $d$-torus of $\GG$ and $\ovS:=\Set{t \in \ovT| \lambda(t) =1 \text{ for all } \lambda \in Y'^\perp}$ a Sylow $d$-torus of $(\ovG,vF)$.
\end{lem}
We now give formulas for the associated Sylow $d$-Levi subgroup and Sylow $d$-normalisers of $(\ovG,vF)$.
Thereby we use the group $\ovN= \NNN_{\ovG}(\ovT)$ from \ref{basicsetting}.

\begin{lem}[Construction of $ \Cent_\ovG (\ovS)$ and $\NNN_{\ovG^{vF}}(\ovS)$, {\cite[3.4, 3.5] {Spaeth_Preprint}}] \label{Konstruktion_reg} 
Let $\ovS$ and $\SSS$ be the tori from the previous lemma. Then $ \Cent_\ovG (\ovS)=\Spann<\ovT, \ovX_\al| \al \in R'>$ with $R':=R\cap Y'^\perp$ and $\NNN_{\ovG^{vF}}(\ovS) = \ovN^{vF}$, if $R'= \emptyset$.
\end{lem}

If $\ovN^{vF}=\spann<\Cent_V(vF),\ovT^{vF}>$ or equivalently $\rho(\Cent_V(v\Gamma))= \Cent_W(\rho(v)\phi)$, the automorphism $v\Gamma$ is called a {\it good Sylow twist} in \cite[Section 3]{Spaeth_Preprint}. By Lemma \ref{bembessis} there exists a good Sylow $d$-twist for every classical group $(\ovG,F)$ and every regular number $d$ of $(\ovG,F)$, whenever $F$ is a standard Frobenius endomorphism.

\section{Maximal extensibility and very good Sylow twists} \label{regred}
In this section we recall some shortening terminology, introduced in \cite[Section 4]{Spaeth_Preprint}.

\begin{defi}[Maximal extensibility] Let $\calL\lhd \calN$ be finite groups and $\chi\in \Irr(\calL)$. An extension of $\chi$ as irreducible character to its inertia group $\I_\calN(\chi)$ in $\calN$ is called a {\it maximal extension of $\chi$ in $\calN$} and in this situation we call $\chi$ {\it maximally extendible}.
If every character $\chi \in \Irr(\calL)$ is maximally extendible, we say that {\it maximal extensibility holds with respect to $\calL\lhd \calN$}.\end{defi}

For example maximal extensibility holds with respect to $\calL\lhd \calN$ whenever $\nicefrac \calN \calL$ is cyclic by \cite[11.22]{Isa}. Extending the standard notion we denote the group \break $\Spann<x\in \calU| \chi^{x}=\chi>$ by $\I_{\calU}(\chi)$ for every subgroup $\calU\leq \calN$ and $\chi \in \Irr(\calL)$ and call $\nicefrac {\I_{\calN}(\chi)}\calL$ the {\it relative inertia group} of $\chi$ in $\calN$. We denote by $\IRR(\calN|\chi)$ the irreducible constituents of $\chi^\calN$ and $\IRR(\calL|\psi)$ the irreducible characters in $\restr \psi|\calL$ for $\psi\in \Irr(\calN)$. Analogously we also generalise the notion of conjugated characters: for a character $\chi\in \Irr(\calU)$ we define the character $\chi^n\in \Irr(\calU^n)$ with $n\in\calN$ by $\chi^n(u)=\chi(u^{n^{-1}})$.

The groups $H$ and $V$ from \ref{EinfV} and good Sylow twists are interesting for the proof of Theorem \ref{Theo1} because of the following remark.

\begin{rem}[{\cite[4.3]{Spaeth_Preprint}}] Let $v\Gamma$ be a good Sylow twist and let maximal extensibility hold with respect to $\Cent_H(v\Gamma) \lhd \Cent_V(v\Gamma)$. Then maximal extensibility holds with respect to $\ovT^{vF}\lhd \ovN^{vF}$, as well. \end{rem}

We conjecture the existence of {\it very good} Sylow $d$-twists, how Sylow twists with the above property were called in \cite[Section 4]{Spaeth_Preprint}.

\begin{conj}\label{Verm2} There exists a very good Sylow $d$-twist of $(\ovG,F)$ for any regular number $d$ of $(\ovG,F)$.
\end{conj}
One can easily verify the existence of very good Sylow twists, when $2\teilt q$, see \cite[Lemma 4.4]{Spaeth_Preprint}. 
So far any possible approach for a proof of the above conjecture would not determine the relative inertia groups $\nicefrac{\I_{\ovN^{vF}}(\la)} {\ovT^{vF}}$ for $\la \in \Irr(T)$, which is relevant in the proof of Corollary \ref{Cor1}. Hence we prefer to prove Theorem \ref{mFreg} more directly.

\section{Maximal extensibility in a special situation} \label{sec_mFreg}
In this section we prove maximal extensibility with respect to some abstract groups. The main statement of this section is Proposition \ref{mFreg}, which we use in the proof of Theorem \ref{Theo1}. Nevertheless in this section all groups are just finite groups having the assumed properties.

\begin{generalass}\label{generalassumption} 
\begin{enumerate}
	\item \label{generalassumptionsa} 
	Let $\wN$, $\wK$ and $\wT$ be finite groups, such that $\wT\lhd \wN$, $\wK\lhd \wN$, $\wT\leq \wK$, $[\wT,\wT]=1$ and $\nicefrac\wN \wK \cong \Sym_j$ for some integer $j$. Let $\wK_1 \lhd \wK$ and $p_r\in \wN$ ($r\in J':=\Lset{1,\ldots,j-1}$) elements, such that with $\wK_{r+1}:=\wK_{r}^{p_r}$ ($r\in J'$) and $\wT_r:=\wK_r\cap \wT$ ($r\in J:=\Lset{1, \ldots,j}$) the following conditions are satisfied: 
	\begin{enumerate}
		\item \label{wKwTkommut} $[\wK_r,\wK_{r'}]\leq \wK_r \cap \wK_{r'}\leq \wT$ and $[\wK_r,\wT_{r'}]=1$ for $r,r'\in J$ with $r\neq r'$,
		\item \label{wKwTprodukt} $\wK=\Spann<\wK_r|r\in J>$ and $\wT=\Spann<\wT_r|r\in J>$,
		\item \label{prkommut}$[\wK_{r'}, p_{r}]=1$ for $r\in J'$ and $r'\in J' \setminus \Lset{r,r+1}$,
		\item \label{wNerzeugtvonSundwK} $S\cap \wK\leq \wT$, $[S,\wK] \leq \wK$ and $S\wK=\wN$ with $S:=\Spann<p_r| r\in J'>$,
		\item \label{zyklK1T1} $\nicefrac {\wK_1}{\wT_1}$ is cyclic.
	\end{enumerate}
\item \label{notationdelta} Let $\delta \in \Irr(\wN)$ with $\delta(S)=1$ and $\ker(\delta_r) \wT_r=\wK_r$ with $\delta_r:=\restr \delta|{\wK_r}$.
\end{enumerate}
\end{generalass}
 
The above implies that the quotient $\nicefrac \wN \wT$ is a wreath product of a cyclic group and the symmetric group $\Sym_j$ on $j$ points.
\begin{ex} Such a situation is the following: Let $\wN:=\Cy_{ab}\wr \Sym_j$, $\wK:=\Cy_{ab}^j$ and $\wT:=\Cy_a^j$ for arbitrary positive integers $a$ and $b$. If we choose $\Cy_{ab}$ as $\wK_1$ and the transposition $(r,r+1)$ in $\Sym_j$ as $p_r$, the above conditions are satisfied.
\end{ex}

For the successive construction of extensions, we use the following maps.
\begin{defi}[(Compatible) extension map]
Let $\calL\lhd \calN$ be finite groups. An {\it extension map for $\calL\lhd \calN$ on $\calI\subseteq \Irr(\calL)$} is a map $\Lambda: \calI\rightarrow \bigcup_{\calL\leq \calU\leq \calN}\Irr(\calU)$ such that for every $\chi \in \calI$ the character $\Lambda(\chi)$ is an extension of $\chi$ to $\I_\calN(\chi)$. If $\calI=\Irr(\calL)$ we call $\Lambda$ an extension map for $\calL\lhd \calN$. 

For a group $\calK$ acting on $\calL$ and $\calN$ the map $\Lambda$ is called {\it $\calK$-equivariant}, if $\Lambda(\la^k)=\Lambda(\la)^k$ for every $k\in \calK$. Let $\mu \in \Irr(\calN)$ with $\mu(1)=1$. Then a map is called {\it compatible with $\mu$}, if $\la \restr\mu|{\calL}\in \calI$ and $\Lambda(\la \restr\mu|{\calL})=\Lambda (\la) \restr\mu|{\I_\calN(\la)}$ for every $\la \in \calI$.\end{defi}

We successively construct extension maps for $\wT_1\lhd \wK_1$, $\wT\lhd \wK$ and $\wT\lhd \wN$. Thereby we make frequent use of the fact, that the extension maps are equivariant. Additionally we prove compatibility with the linear character $\delta$, as this is needed at the end of this section.

\begin{lem}\label{defTheta'} There exists a $\wK_1$-equivariant $\delta_1$-compatible extension map $\Theta'$ for $\wT_{1} \lhd \wK_{1}$.\end{lem}

\begin{proof} Let $\wK_{1,\delta}:=\ker(\delta_1)$ and $\wT_{1,\delta}:=\wK_{1,\delta}\cap \wT_1 $. According to \ref{zyklK1T1} the group $\nicefrac {\wK_{1}}{\wT_{1}}$ and hence $\nicefrac {\wK_{1,\delta}}{\wT_{1,\delta}} \cong \nicefrac{\wK_{1,\delta} \wT_1}{\wT_1}$ are cyclic and maximal extensibility holds with respect to $\wT_{1,\delta} \lhd \wK_{1,\delta}$. 
We can construct a $\wK_{1,\delta}$-equivariant extension map $\Lambda''$ for $\wT_{1,\delta} \lhd \wK_{1,\delta}$ by choosing a representative set in the $\wK_{1,\delta}$ -orbits in $\Irr(\wT_{1,\delta})$. The extension map $\Lambda'$ with $\Lambda'(\chi) (kt)=\Lambda''(\restr \chi|{\wT_1})(k) \chi(t)$ for $k \in \I_{\wK_{1,\delta}}(\chi)$ and $t\in \wT_{1,\delta}$ is well-defined and has the desired properties by definition.\end{proof}

Before constructing a similar map we repeat some well-known facts about central products.
\begin{notation} Recall, a finite group $A$, generated by the subgroups $A_r\leq A$ with $[A_r, A_{r'}] =1$ for $r, r' \in J$ and $r\neq r'$ is called the {\it central product of the subgroups $A_r$} and denoted by $A=A_1. \cdots . A_j$. Every irreducible character $\lambda$ of $A$ satisfies $ \lambda (\prod_r a_r)= \prod_r \lambda_r(a_r) $ for $a_r\in A_r$ and $\lambda_r\in \Irr(A_r|\lambda)$, where $\Irr(A_r|\lambda)$ denotes the set of irreducible constituents of $\restr \lambda|{A_r}$, see \cite[4.20]{Isa}. We denote such a character by $\la_1. \cdots .\la_j$ and by $\la_1\times \cdots \times \la_j$ if $A$ is the direct product of the groups $A_r$.
\end{notation}

In our context $\wT$ is a central product: because of $[\wK_r, \wT_{r'}]= 1$ ($r,r'\in J$, $r\neq r'$) from \ref{wKwTkommut} the group $\wT$ is the central product of the groups $\wT_r$ ($r\in J$).

For the construction of an extension map for $\wT\lhd \wK$ we restrict the character to the subgroups $\wT_r$ and obtain extensions of these characters by using $\Theta'$. For gluing together the extensions we need the following assumption on the groups $\wK_r$.

\begin{hypo}\label{hyp_inkommut}
If $j\geq 2$, then $Z:=\wT_1\cap (\prod_{r=2}^j \wT_r)$ is a subgroup of the centre $\Zent(\wN)$ of $\wN$ and $[\I_{\wK_r}(\la), \wK_{r'}]\leq \ker(\la)\cap Z$ hold for any $r, r'\in J$ with $r\neq r'$ and any $\la\in \Irr(\wT_r)$.
\end{hypo}

\begin{lem} \label{defTheta} For $\wT \lhd \wK$ an $\wN$-equivariant $\restr \delta|{\wK}$-compatible extension map $\Theta$ exists.
\end{lem}
\begin{proof} Because of $\wT=\wT_1.\cdots .\wT_j$ every character $\la \in \Irr(\wT)$ coincides with $\la_1. \cdots .\la_j$ with $\la_r:=\restr \la|{\wK_r}$ ($r\in J$).
Using the isomorphisms $ \iota_{1}:= \operatorname{id}_{N_{1}}$ and $\iota_{r}: \wK_{r} \rightarrow \wK_{1}$ with $\iota_{r}(x):=\iota_{r-1}(x^{p_{r-1}})$ for $x \in \wK_{r}$ and $2\leq r$ one can associate to every character $\mu \in \Irr(\wT_r)$ a character $\mu\circ {\iota'}_r^{-1} \in \Irr(\wT_1)$ with $\iota_r':=\restr\iota_r|{\wT_r}: \wT_r\rightarrow \wT_1$ and $(\mu\circ {\iota'}_r^{-1}) (x)= \mu({\iota'}_r^{-1}(x))$ for $x\in \wT_r$. Using $\iota_r$ one associates to $\Lambda'(\la\circ{\iota'}_r^{-1})$ a character of $\Irr({\I_{\wK_r}(\la)})$, which we define to be $\Lambda_r(\mu)$. By this definition $\Lambda_r$ is a $\wK_r$-equivariant $\delta$-compatible extension map for $\wT_r\lhd \wK_r$.

Because of \ref{wKwTkommut} the inertia group $\I_{\wK}(\la)$ coincides with $\I_{\wK_1}(\la_1) \ldots \I_{\wK_j}(\la_j)$. Using Hypothesis \ref{hyp_inkommut} the group $\nicefrac{\I_{\wK}(\la)} {Z_\la}$ with $Z_\la:=\ker(\la)\cap Z$ is the central product of the groups $\nicefrac{\I_{\wK_r}(\la_r)} {Z_\la}$ ($j \in J$). Hence the character $\Theta(\la):= \Lambda_1(\la_1). \cdots . \Lambda_j(\la_j)$ is well-defined.

The thereby defined extension map $\Theta$ is $\delta$-compatible and $\wK$-equivariant as the maps $\Lambda_r$ are $\wK_r$-equivariant and Hypothesis \ref{hyp_inkommut} holds.

It remains to check if $\Theta$ is also $S$-equivariant. The character $\la':=\la^{p_r}$ coincides with $\la_1. \cdots. \la_{r+1}^{p_r}. \la_{r}^{p_r}. \ldots. \la_j$. Hence the restriction of $\Theta(\la)^{p_r}$ and $\Theta(\la^{p_r})$ to $\I_{\wK_{r'}}(\la')$ ($r'\notin \Lset{r,r+1}$) coincides. By definition of $\iota_{r'}$ we may calculate $\restr\Theta(\la')|{\I_{\wK_{r}}(\la')}= \Theta'_r(\la_{r+1}^{p_r})=
\Theta'(\la_{r+1} \circ{\iota'}_{r+1}) \circ \iota_r^{-1}=
\left( \Theta_{r+1}(\la_{r+1}) \right)^{p_r}$. The characters $\Theta(\la)^{p_r}$ and $\Theta(\la^{p_r})$ restricted to $\I_{\wK_r}(\la')$ coincide. Analogously one can prove $\restr\Theta(\la)^{p_r}|{\I_{\wK_{r+1}}(\la')}
=\restr\Theta(\la^{p_r})|{\I_{\wK_{r+1}}(\la')}$ and this implies $\Theta(\la)^{p_r}=\Theta(\la^{p_r})$.
\end{proof}

In the next step we construct an extension map on a subset $\calI$ of the irreducible characters, which are distinguished by their easily described inertia subgroup.

\begin{defi}[The subset $\calI$] \label{DefI}
Let $\calI\subseteq\Irr(\wT)$ consist of all characters $\lambda\in \Irr(\wT)$ fulfilling $\I_{\wN}(\la)=\I_{\wK}(\la) S_I$ for some $I\subseteq J'$, where $S_I=\Spann<p_i|i \in I>$ .

Let $\la=\la_1\circ {\iota'}_1.\cdots.\la_j\circ {\iota'}_j\in \Irr(\wT)$ with $\la_r\in \Irr(\wT_1)$. According to the action of $\wN$ on $\wT$, a character $\la \in \Irr(\wT)$ lies in $\calI$, if and only if the following conditions are satisfied:
\begin{enumerate}
\item Two $\wK_1$-conjugate characters in $\Lset{\la_1,\ldots ,\la_j}$ are equal.
\item "Characters from the same $\wK_1$-orbit in $\Irr(\wT_1)$ are in neighbouring positions", that is the equation $\la_i=\la_{i'}^k$ for some $k \in \wK_1$ and $i,i'\in J$ implies $\lambda_{i''}\in \Lset{\la^k| k \in \wK_1}$ for $i\leq i''\leq i'$.
\end{enumerate}
Because of the first condition the characters $\la_{i''}$ from the second condition are all equal.
One can see $\calI$ as the intersection of the following two sets corresponding to each of the above conditions, namely the $\calI_1$ defined by the first condition and $\calI_2$ defined by the other one. By the action of $\wN$ on $\wT$ we know that conjugation with elements of $S$ stabilises $\calI_1$ and analogously $\wK$ acts on $\calI_2$.
\end{defi}

These characters are maximally extendible in $\wN$ and a maximal extension is be well-defined, if one assumes some additional properties for the elements $\Lset{p_r| r \in J'}$.

\begin{hypo}[Properties of $\Lset{p_r|r\in J'}$]\label{hypo_pr1}
For every $I\subseteq J'$ and $S_I:=\Spann<p_r|r \in I>$ the equation $S_I\cap \ovT= \Spann<p_r^{2}| r \in I> $ holds and $p_r^2 = t (t^{p_r})^{-1}$ for $r\in J'$ and some $t\in \wT_r$.
\end{hypo}

This enables us to prove the following lemma. 

\begin{lem} \label{lem_FortsetzungI} There exists a $\delta$-compatible extension map $\Lambda$ for $\wT\lhd \wN$ on $\calI$.\end{lem}

\begin{proof} Let $\la=\la_1\circ {\iota'}_1.\cdots.\la_j\circ \iota_j'$.
By definition the inertia group of $\lambda \in \calI$ fulfils $\I_{\wN}( \lambda)= \I_{\wK}( \lambda) S_I$ where $I$ is the set of all $r\in J'$ with $ \lambda_{r}= \lambda_{r+1}$.

For the later construction we check $S_I\cap \wT \leq \ker(\la)$. By the first part of Hypothesis \ref{hypo_pr1} it suffices to prove $\la(p_r^2)=1$ for $r\in I$. But this is implied by the later part of Hypothesis \ref{hypo_pr1}, as $\la(p_r^2)=\la(t (t^{p_r})^{-1})=\la_r(\iota_r(t)) \la_{r+1}(\iota_{r+1}(t^{p_r}))^{-1}$. Thereby the last term is equivalent to $\la_{r+1}(\iota_{r}(t^{p_r^2}))^{-1}$.
Using the definition of $I$ this implies $\la(p_r^2)=1$. Because of $S_I\cap \wT \leq \ker(\lambda)$ the character $\Theta(\la)$ can be extended to $\widetilde \la$ with $\restr \widetilde \la|{S_I}=1$.

Before proving that $\Lambda$ is $\delta$-compatible we check $\restr \delta|{\wT} \la\in \calI$ for $\la \in \calI$. The groups of $\I_{\wN}(\restr \delta|{\wT}\la)$ and $\I_{\wN}(\la)$ coincide, because the characters are linear and $\delta $ is defined on $\wN$. As $\Theta$ is $\delta$-compatible by Lemma \ref{defTheta} and $\delta(S_I)=1$ by Assumption \ref{notationdelta}, $\Theta $ is compatible with $\delta$.
\end{proof}

In order to obtain an $\wN$-equivariant extension map for $\wT\lhd \wN$, we describe $\Lambda$ on $\wN$-conjugated characters in $\calI$. For showing the $\wN$-equivariance we need an additional assumption on the elements $p_r$.

\begin{hypo}\label{hyp_pr2}
The elements $\Lset{p_1, \ldots, p_{j-1}}$ satisfy $ [p_r,kk^{p_r}]=1$ for every $k\in \wK_r $. The equation $S_I^s \wT= S_{I'} \wT$ for $s\in S$ and $I,I'\subseteq J'$ implies $S_I^s=S_{I'}$.
\end{hypo}

\begin{lem}\label{NormalformC}
Let $s\in S$, $k \in \wK$, $x:=ks$ and $\lambda \in \calI$ with $\lambda^{x}\in \calI$. Then $\lambda^{k} \in \calI$, $\Lambda(\lambda^{k})= \Lambda(\lambda)^{k}$ and $\Lambda(\la^x)=\Lambda(\la)^x$.
\end{lem}
\begin{proof} Let $\la':=\la^{k}$ and $\la'':=\la^{x}$. By assumption we have $\la \in \calI$ and $\la'' \in \calI$. By Definition \ref{DefI} the equation $\calI=\calI_1\cap\calI_2$ holds for the sets $\calI_1$ and $\calI_2$. Using $\la'=\la^k $ and $\la'=\la''^{(s^{-1})}$ 
we deduce $\la' \in \calI_1\cap \calI_2=\calI$ from $\calI_1^{s^{-1}}=\calI_1$ and $\calI_2^k=\calI_2$.

As we have defined $\Lambda$ by using the $\wN$-equivariant extension map $\Theta$, it suffices for proving $\Lambda(\lambda^{k})= \Lambda(\lambda)^{k}$ to check the equation $\Lambda(\lambda^{k})(p_r)= \Lambda(\lambda)^{k} (p_r)$ for $r \in I$. Let $k=\prod_r k_r$ with $k_r\in \wK_r$ and $\la=\la_1\circ {\iota'}_1.\ldots.\la_j\circ{\iota'}_j$. By definition of $\calI$ the elements $k_{r}$ and $k_{r+1}$ ($r\in I$) satisfy $\la_r^{\iota_r(k_r)}=\la_{r+1}^{\iota_{r+1}(k_{r+1})}$. Hence we may assume $k_{r+1}=k_r^{p_r}$. For such an element Hypothesis \ref{hyp_pr2} implies $\la(p_r^k)=\la(p_r)$ after some short calculations.

For the remaining statement $\Lambda(\la^x)=\Lambda(\la)^x$ we may assume $x\in S$.
Let $I,I'' \subseteq J'$ such that $\I_{\wN}(\la)=\I_{\wK}(\la) S_I$ and $\I_{\wN}(\la'')=\I_{\wK}(\la)S_{I''}$. By Hypothesis \ref{hyp_pr2} this implies $S_I^x=S_{I''}$. Hence $\Lambda(\la'')$ and $\Lambda(\la)^x$ coincide on $S_I''$. 
As $\Lambda$ is defined using the $\wN$-equivariant extension map $\Theta$, this shows already $\Lambda(\la'')=\Lambda(\la)^x$.
\end{proof}

The above properties enable us to construct an extension map for $\wT\lhd \wN$.

\begin{lem}\label{Endelemma}There exists an $\wN$-equivariant $\delta$-compatible extension map for $\wT\lhd \wN$.\end{lem}

\begin{proof} We extend the domain of $\Lambda$ by $\Lambda(\la^x\restr \delta^i|{\wT}):=\Lambda(\la)^x\restr \delta^i|{\I_{\wN}(\la^x)}$ for all $x\in \wN$ and integers $i$. By definition of $\cI$ every irreducible character on $\wT$ can be expressed in such a way. The map is well-defined according to the previous lemma. Because of the properties on $\cI$ the map is $\wN$-equivariant and compatible with $\delta$.
\end{proof}

Using this extension map $\Lambda$ we prove maximal extensibility with respect to $\ker(\restr\delta|\wT)\lhd \ker(\delta)$. 

\begin{prop}\label{mFreg} 
If the general assumption \ref{generalassumption} and the hypotheses \ref{hyp_inkommut}-\ref{hyp_pr2} hold, then maximal extensibility holds with respect to $T \lhd N$, where $N:=\ker (\delta)$ and $T:=\ker (\delta)\cap \wT$.\end{prop}

\begin{proof}
Let $\mu \in \Irr(T)$ with $T$ and $\la\in \IRR(\wT|\mu)$. The character $\Lambda(\la)$ is a maximal extension of $\la$ in $\wN$ by definition. Now we have to extend $\widetilde \mu:=\restr \Lambda(\la) |{\I_N(\la)}$ to $\I_N(\mu)$.

Let $n\in \I_N(\mu)$. The character $\la^n$ is an extension of $\mu$. By \cite[6.17]{Isa} there exists an integer $i$ such that $\la^x=\la\delta_0^i$ with $\delta_0:=\restr \delta|{\wT}$. Because of $\delta\in \Irr(\wN)$, the equation $\I_{\wN}(\la^n)= \I_{\wN}(\la\delta_0^i)= \I_{\wN}(\la)$ holds. This implies $\I_{\wN}(\la)\lhd \I_{\wN}(\mu)$. The character $\widetilde \mu$ is invariant in $\I_{N}(\mu)$, as $\Lambda$ is $\wN$-equivariant and 
 $\widetilde \mu^n=\restr \Lambda(\la^n)|{\I_N(\la)}= \restr\Lambda(\la \delta_0^i)|{\I_N(\la)}= \restr\Lambda(\la) \delta^i|{\I_N(\la)}= \widetilde \mu$ holds.

The group $\I_{\wN}(\mu)$ acts on $\Irr(\wT|\mu)$ such that $\Set{\mu^n|n \in \I_{\wN}(\mu)}=\Lset{\mu \restr \delta^{ii_0}|{\wT}}$ for some integer $i_0$, where $i$ runs through all integers. 
Hence $\I_{N}(\mu)/\I_{N}(\la)$ is cyclic and the character $\mu$ can be extended to $\I_{N}(\la)$ according to \cite[11.22]{Isa}.
\end{proof}

For the proof of Corollary \ref{Cor1} we compute the occurring relative inertia groups.

\begin{rem}\label{rem_relI}
Let us identify $\nicefrac {\wN}{\wT}$ with $\overline K\rtimes \Sym_j$, where $\overline K_r:= \nicefrac{\wK_r}{\wT_r}$ and $\overline K:=\nicefrac{\wK}{\wT}=\spann<\overline K_r>$ and $\Sym_j\cong \rho(S)$ permutes the factors of $\overline K$. The relative inertia group of $\la \in \calI$ in $\wN$ is a subgroup of $\overline K\rtimes \Sym_j$ of the form $(A_1 \times\ldots \times A_j )\rtimes U$, where $A_r\leq \overline K_r$ and $U$ is a direct product of symmetric groups $\Sym_{I_1} \times \ldots \times \Sym_{I_j}$, such that $I_1, \ldots, I_j$ are (possible empty) subsets of $\Lset{1, \ldots , j}$ with trivial intersections and $A_r\cong A_{r'}$ for $r,r' \in I_{i''}$. By definition of $\calI$ the relative inertia subgroup of $\la \in \Irr(\wT)$ is isomorphic to such a group. 

Let $\la \in \Irr(T)$. Then the relative inertia subgroup $\overline I$ of $\la$ in $N$ has a normal subgroup $\overline I_0$, which is isomorphic to $(A_1 \times\ldots \times A_j )\rtimes U$ for appropriate groups $A_i$ and $U$. Furthermore the quotient $\nicefrac {\overline I }{\overline I_0}$ is cyclic and its cardinality divides the order of $\delta$. \end{rem}

\begin{proof} This statement can be deduced from the proofs of Lemma \ref{lem_FortsetzungI} and Proposition \ref{mFreg}.\end{proof}

\section{Towards an Application of Proposition \ref{mFreg}}\label{sec_towards_application}
For the proof of Theorem \ref{Theo1} we want to apply the Proposition \ref{mFreg}. For this purpose we introduce groups $\wN$, $\wK$, $\wT$ and $\wT_1$ and elements $\Lset{p_r}$ such that they satisfy the Assumption \ref{generalassumptionsa}. Afterwards we have to ensure that with this choice the hypotheses \ref{hyp_inkommut}-\ref{hyp_pr2} hold. Only some of the desired properties can be checked without referring to the actual root system. The aim of this section is to present candidates for the groups and elements. Further we prove their properties, as far this can be done without referring to the actual root system.

For the proof we assume that besides the notion of Section \ref{sec_not} a further root system $\ovR$ and a Frobenius endomorphism $\oF$ is given in the following way.

\begin{ass} \label{ass_ovR_oF}
We assume a given root system $\ovR$ of type $\tA_{l-1}$, $\tB_l$ or $\tC_l$, with the following properties:
\begin{itemize}
	\item $R$ is a parabolic root system of $\ovR$ or $\ovR$ itself,
	\item $\oF:\ovovG\rightarrow \ovovG$ is a Frobenius endomorphism with $\restr \oF|{\ovG}=F$, where the groups $\ovovG$ and $\ovovX_\al$ ($\al\in \ovR$) are associated to $\ovR$ as in Section \ref{sec_not} and $\ovG$ is identified with $\Spann<\ovovX_\al|\al \in R>$ by Remark \ref{einbett}.
\end{itemize}

With the same definitions as in Section \ref{sec_not} we also associate to $\ovR$ the groups $\ovovN$, $\ovovT$ and $\ovW$, and additionally define $\ovrho:\ovovN \rightarrow \ovW$ to be the canonical epimorphism. Because of $\ovG=\Spann<\ovovX_\al|\al \in R>$ the inclusions $\ovT\leq \ovovT$, $\ovN \leq \ovovN$ and $W\leq \ovW$ hold, furthermore $\restr \ovrho|{\ovN}=\rho$.

The Frobenius endomorphism $\oF$ can differ from the Frobenius endomorphisms given in Notation \ref{setgen}. Let $\ovphi$ be the endomorphism induced by \oF on the cocharacter lattice $\ovY$ on \ovovT, which then automatically satisfies $\restr \ovphi|{Y}=\phi$.
\end{ass}

For defining the candidate for $\wK$, we need the following description of classical root systems.

\begin{ass}[The roots and the Weyl group]

\label{ass_classical_roots}
Let $l\geq 2$ be an integer and $\Lset{e_i}_{1\leq i \leq l}$ an orthonormal basis of the vector space $\CC^l$. 
Let $R$ and $\ovR$ be \[\Set{e_i -e_{i'}|1\leq i,i'\leq l \text{ with }i\neq i'}\text{ or } \Set{\pm a e_i, \pm e_i \pm e_{i'}|1\leq i,i'\leq l \text{ with }i\neq i'}\] for some $a\in \Lset{1,2}$, the roots from \cite[Section 5]{Springer74}.

In this framework $\spann<\ovW,\ovphi>$, seen as subgroup of the automorphisms of $\ovY$, acts on $\Lset{\pm e_1, \ldots, \pm e_l}$. Hence there exists a morphism $f: \spann<\ovW,\ovphi> \rightarrow \spann<(-1,+1)>\wr \Sym_l<\Syml$ with $x(e_i)=\sgn(\sigma(i)) e_{\betrag{\sigma(i)}}$ for $1\leq i \leq l$, $x\in \spann<\ovW,\ophi>$ and $\sigma:=f(x)$. By $\Syml$ we denote the symmetric group acting on $\Lset{\pm 1, \ldots, \pm l}$. This map is injective on $\ovW$ and one can deduce from $f$ the morphism $\ovf:\spann<\ovW,\ovphi> \rightarrow \Sym_l$ where $\ovf(x)(i)=\betrag{f(x)(i)}$.
\end{ass}

In order to apply Proposition \ref{mFreg} in the proof of Theorem \ref{Theo1} we define groups $\wN$, $\wK$, $\wK_1$ and $\wT$ and a character $\delta\in \Irr(\wN)$ which satisfy the assumptions from Section \ref{sec_mFreg} with $\ker(\restr\delta|\wT)$ being a Sylow Levi $d$-subgroup and $\ker(\delta)$ being the associated Sylow normaliser. Hence we use a Sylow twist for the definition of the groups. Thereby we have to restrict ourselves to Sylow twists with certain properties.

\begin{ass}[The Sylow $d$-twist $v\Gamma$]\label{ass_Sylowtwist}
Let $d$ a regular number of $W\phi$ and $v\Gamma\in V\Gamma$ a Sylow $d$-twist of $(\ovG,F)$, such that the following conditions are satisfied:
\begin{enumerate}
	\item \label{ass_Sylowtwista}with $\osigma:=\ovf(\rho(v) \phi)$ the \osigma-orbits $\calB_r$ ($r\in J:=\Lset{1, \ldots,j}$) on $\{1,\ldots, l\}$ have the same length,
	\item \label{ass_Sylowtwistb}$J$ is a transversal of the orbits and $r\in \calB_r$,
	\item \label{ass_Sylowtwistc}for $w=\rho(v)$ the group $\Cent_{W_{R_1}} ( w\phi)$ is cyclic, where $R_{r}:=R\cap \Spann<e_{i}|i \in \calB_{r}>$,
	\item \label{ass_Sylowtwistd}$\rho(\ovN^{vF})$ is isomorphic to $\Cent_{W_{R_1}}(w)\wr \Sym_\calB$, where $\calB$ is the set of the orbits $\calB_r$ ($r \in J$), more specifically $\Cent_{\prod_{r\in J}W_{R_r}}(f(\rho(v)\phi))$ is normal in $\rho(\ovN^{vF})$ and has a complement isomorphic to $\Sym_\calB$, and 
	\item \label{ass_Sylowtwiste}$x^{(v F)^{\frac l j}}=x$ for every $x \in V$.
\end{enumerate}
\end{ass}
As the roots are as in Assumption \ref{ass_classical_roots} the root systems $R_r$ and $\ovR_r:=\ovR\cap \Spann<e_{i}|i \in \calB_{r}>$ are of the same type as $R$ or $\ovR$, respectively, but have a smaller rank in general.

\begin{lem}\label{lem1}
For every root subsystem $R'$ in $ \ovR$
let $\ovovN_{R'}:=\Spann<n_\al(t)|\al \in R', \,\, t\in \ovF_q^*>\leq \ovovG$ and $\ovovT_{R'}:=\Spann<h_\al(t)|\al \in R' , \,\, t\in \ovF_q^*>\leq \ovovG$.
 We also denote $\ovovN_{R'}$ by $\ovN_{R'}$ and $\ovovT_{R'}$ by $\ovT_{R'}$, if $R'\subseteq R$. Furthermore let $\ovovT_r:=\ovovT_{\ovR_r}$, $\ovK_r:=\ovN_{R_r}\ovovT_r$, $\ovK:= \ovN_{\bigcup_r R_r}\ovovT$ and $\calZ\leq \Zent(\ovG)\cap \ovT_1$. Let $\wN:=\Lang^{-1}(\calZ)\cap \ovN\ovovT$, $\wK:= \Lang^{-1}(\calZ)\cap \ovK$, $\wK_1:= \Lang^{-1} (\calZ) \cap \ovK_1$, $\wT=\Lang^{-1}(\calZ)\cap \ovT$, where $\Lang: \ovovG\rightarrow \ovovG$ is defined by $x\mapsto x^{v\overline F} x^{-1}$. Let $p_{r}:=\prod_{i=0}^{\frac l j -1} n_{\beta_r}(1)^{(v\Gamma)^i}$ with $\beta_r:=e_{r+1}-e_{r}$ ($ r \in J'$).

Then the above groups satisfy Assumption \ref{generalassumptionsa} if $\ovovT=\ovovT_1 \cdots \ovovT_j$ and additionally - in case of $j\geq 2$ -the equation $\calZ=\ovT_1 \cap \ovT_2\cdots \ovT_j$ holds.
\end{lem}

\begin{proof}First we observe that by definition $\wT\lhd \wN$ and $\wT\leq \wK$. Afterwards we check $p_r\in \wN$ and if all equations and inclusions from Notation \ref{generalassumptionsa} are satisfied.

As all orbits $\calB_r$ have equal length by \ref{ass_Sylowtwista}, $ (\rho(v)\phi)^{i'}(e_{r}) \in \Lset{\pm e_r}$ for some $r\in J$ and some integer $i'$ implies $(\rho(v)\phi)^{i'}(e_{r'}) \in \Lset{\pm e_{r'}}$ for all $r'\in J$. 
For $r\in J'$ and $1< i \leq \frac l j-1$ the roots $\beta_r$ and $ (\rho(v)\phi)^i(\beta_r)$ are orthogonal and 
$\beta_r \pm (\rho(v)\phi)^i(\beta_r) \notin R$, according to the roots given in Assumption \ref{ass_classical_roots}. This implies, according to the Steinberg relations, $[n_{\beta_r}(1), n_{\beta_r}(1)^{(v\Gamma)^{i}}]=1$ for $r\in J'$ and $1\leq i \leq\frac l j-1$. From this and $x^{(v\Gamma)^{\frac l j }}=x$ for all $x\in V$ one obtains $p_i\in V^{vF}$, as the standard Frobenius endomorphism acts trivially on $V$.

Then we have to check $\wK\lhd \wN$. It suffices to prove $K\lhd N$ where $K:=\ovN_{\bigcup_r\ovR_r}^{vF}$ and $N:=\ovN^{vF}$. Because of Assumption \ref{ass_Sylowtwist} we have 
$\Cent_{W_{\cup_{r\in J}R_r}} (f(\rho(v)\phi)) = \rho(\prod_{r\in J}\ovK_r^{vF}) \lhd \rho(\ovN^{vF})$. As the kernel of $\restr\rho|N$ lies in $K$ this implies $K\lhd N$. It is clear that $\wT$ is abelian and the quotient $\nicefrac \wN \wK$ is isomorphic to $\Sym_j$.

Each root $\al \in \bigcup_{r\in J}R_r$ either lies in $R_1$ or is orthogonal to all roots in $R_1$. Hence $\ovK_1$ is a normal subgroup of $\ovK$. This implies $\wK_1\lhd \wK$.

As $p_r$ permutes the orbits in $\calB$ these elements also act on the associated root systems $R_r$ and groups $\ovK_r$ according to the Steinberg relations. Hence the groups $\wT_r$ and $\wK_r$ ($r\in J$) defined by \ref{generalassumptionsa} satisfy $\wK_r=\Lang^{-1}(Z)\cap \ovK_r $ and $\wT_r=\Lang^{-1}(Z)\cap \ovT_r $.

The fact $\al\perp\beta$ for $\al \in \ovR_r$ and $\beta \in \ovR_{r'}$ ($r,r'\in J$, $r\neq r'$) implies $[\wK_r, \wK_{r'}]\leq \wK_r\cap \wK_{r'}$ and $[\wK_r,\wT_{r'}]=1$.

Using the Steinberg relations, we see that $p_r$ ($r\in J'$) lies in $\ovN_{R_{r,r+1}}$, where $R_{r,r+1}=\Set{\pm (e_{i} - e_{i'})|i \in \calB_{r}, \, i'\in \calB_{r+1}}$. Because of the description of $R$ in Assumption \ref{ass_classical_roots} one observes $\al+\beta \notin R$ for $\al \in \ovR_{r'}$ and $\beta\in R_{r,r+1}$. This implies $[\wK_{r'},p_r]=1$ and \ref{prkommut}.

The element $p_r$ maps to $\prod_{i=1}^{\frac l j}(\sigma^i(r),\sigma^i(r+1)) (-\sigma^i(r),-\sigma^i(r+1))$ via $f\circ \rho$. Computations in $\Syml$ show that $\rho(S)$ with $S:=\Spann<p_r|{r\in J'}>$ is a complement of $\rho(\wK)$ in $\rho(\wN)$. This implies $\wN=\wK S$ and $S\cap\wK\leq \wT$. As the elements $p_r$ ($r\in J'$) normalise $\wK$, the group $S$ does it also.

Like before $\nicefrac {\wK_1} {\wT_1}$ is isomorphic to $\nicefrac {\ovN_{R_1}^{vF}}{\ovT_{R_1}^{v\overline F}}$ because of the Theorem of Lang. By \cite[3.3.6]{Car2} this group is isomorphic to $\Cent_{W_{R_1}}(\rho(v)\phi)$, which is cyclic according to Assumption \ref{ass_Sylowtwist}.

Because of the assumptions $\ovovT=\ovovT_1 \cdots \ovovT_j$ and $\calZ=\ovT_1 \cap \ovT_2\cdots \ovT_j$ the group $\wT$ coincides with $\wT_1\cdots \wT_j$. The groups $\wK$ and $\Spann<\wK_r|r\in J>$ have the same image under $\rho$ because of \cite[3.3.6]{Car2}. Using $\wT=\wT_1\cdots \wT_j$ this implies \ref{wKwTprodukt}.
\end{proof}

With the above assumptions on $\calZ$ and $\ovT$ the elements already have some of the properties from the hypotheses \ref{hypo_pr1} and \ref{hyp_pr2}.
\begin{lem}\label{lem2}
\begin{enumerate}
\item \label{lem2a} For every $I\subseteq J'$ and $S_I:=\Spann<p_r|r \in I>$ the equation $S_I\cap \wT= \Spann<p_r^{2}| r \in I> $ holds.
\item \label{lem2b} The equation $S_I^s \wT= S_{I'} \wT$ for $s\in S$ and $I,I'\subseteq J'$ implies $S_I^s=S_{I'}$.
\end{enumerate}
\end{lem}
\begin{proof}
Using \cite[9.3.2]{Springer} calculations show that the elements $p_r$ satisfy the braid relations, hence there exists an epimorphism $\tau':\B\rightarrow S$ with $\bb_r\mapsto p_r$, where \B denotes the braid group of type $\tA_{j-1}$ with generators $\Lset{\bb_1, \cdots , \bb_{j-1}}$. The morphisms $\rho\circ \tau': \B\rightarrow \rho(S)$ is the well-known epimorphism of the braid group onto the symmetric group on $J$, hence the kernel of $\restr \rho\circ\tau|{\Spann<\bb_r|r\in I>}$ is generated by ${\Spann<\bb_r^2|r\in I>}$. This implies $S_I\cap \wT= \Spann<p_r^{2}| r \in I> $.

The map $p_r \mapsto n_{\beta_r}(1)$ for $r\in J'$ defines an isomorphism $\omega$ between $S$ and $ \widetilde V$ between $S$ and ${\widetilde V:=\Spann< n_{\beta_r}(1) | r \in J'>}\leq \ovG$. The group $\widetilde V$ is the extended Weyl group of type $\tA_{j-1}$. A presentation of $\widetilde V$ can be found in \cite{Tits}. For proving that $\omega $ is a isomorphism one shows that the elements $p_r$ fulfil the relations, and that the orders of both groups coincide.

Now we concentrate on the second part of the statement. Let $n\in N_{S}$ and $I,I'\subseteq J'$ such that $S_I^s \wT= S_{I'} \wT$. Using $\omega $ and the Steinberg relations gives $ \omega(p_r^s)= n_{\beta_r}(1)^{\omega(s)}=n_\beta'(1)$, where $\beta'$ is a root in the root system with simple roots $\Set{\beta_{r'}| r' \in I'}$. This implies $ \omega(p_r^s)\in \omega(S_{I'})$ and $S_I^s=S_{I'}$.
\end{proof}

The conditions on $\calZ$ and $\ovT$ from Lemma \ref{lem1} are linked tightly to the behaviour of certain lattices.
\begin{lem}\label{lem3_lattices}
Let $j\geq 2$, $p$ the prime and $q$ the power of $p$ from Section \ref{sec_not}. Assume the following:

\begin{itemize}
\item $2\al^\vee \in \bigoplus_{r=1}^j \ZZ \overline R_r^\vee$ for all $\al \in \overline R$, and
\item $\mathcal Q:=\nicefrac{\spann< \ZZ \overline R_1^\vee , (q-1) \ZZ\overline R^\vee> \cap \Spann< \ZZ \overline R_{r}^\vee , (q-1) \ZZ \overline R^\vee | 1\neq r\in J>}{(q-1)\ZZ \overline R^\vee}$ is isomorphic to $\calZ$. The isomorphism type of $\calQ$ depends on $p$, but not on $q$.
\end{itemize}
Then the equation
$\ovovT=\ovovT_1 \cdots \ovovT_j$ holds and additionally - in case of $j\geq 2$ -the equation $\calZ=\ovT_1 \cap (\ovT_2\cdots \ovT_j)\leq \Zent(\wN)$ is true.
\end{lem}
\begin{proof}
For proving $\ovovT=\ovovT_1 \cdots \ovovT_j$ let $x \in \ovovT$. There exists an integer $a$ such that $x=x^{p^a}$ and $x \in \ovovT^{\widetilde F}$, where $\widetilde F:\ovG\rightarrow \ovG$ is the standard Frobenius endomorphism associated to $q':=p^{2a}$.

According to \cite[Proposition 3.2.2]{Car2} the groups $\ovovT^{\widetilde F}$ and $\nicefrac{\ZZ \overline R^\vee}{(q'-1) \ZZ \overline R^\vee} $ are isomorphic: the isomorphism $\omega: \ovovT^{\widetilde F}\rightarrow \nicefrac{\ZZ \overline R^\vee}{(q'-1) \ZZ \overline R^\vee} $ is given by $ h_{\beta^\vee}(\zeta)\mapsto \beta + (qÄ-1) \ZZ \overline R^\vee$ ($\beta \in R^\vee $), where $\zeta \in \ovF_q$ is a primitive $(q'-1)$th root of unity. There exists an $v\in \ZZ \overline R^\vee$ with $ \omega(x)=v+ (q'-1)\ZZ \overline R^\vee$. From $x^{p^a}=x$ we obtain $(p^a-1)v\in (q'-1)\ZZ R^\vee$ and $v \in \left(\frac{q'-1}{p^a-1} \right)\ZZ\overline R^\vee$. For odd $q$ this shows $ v \in 2\ZZ \overline R^\vee$. For even $p$ we replace $v$ by $p^av$. Using the first assumption we get $v\in 2\ZZ \overline R^\vee\leq \bigoplus_{r\in J} \ZZ \overline R_r^\vee$ and $\omega(x) \in \Spann<\omega(\ovT_r^{\widetilde F})|r \in J>$. This implies $\ovovT \leq \Spann<\ovT_r| r\in J>$.

For proving $\calZ:=\Zent(\ovG)\cap \ovT_1=\ovT_1 \cap (\ovT_2\cdots \ovT_j)$ it suffices to check $\calZ \leq \ovT_1 \cap (\ovT_2\cdots \ovT_j)$. By definition $\calZ\leq \ovT_1$, conjugating with elements of $S$ maps $\ovT_1$ to $\ovT_r$($r \in J$). Hence $\calZ \leq \bigcap_{r\in J} \ovT_r$. This implies $\calZ^{F_0}\leq (\ovT_1 \cap (\ovT_2\cdots \ovT_j))^{F_0}= \omega^{-1}(\calQ)$, where $F_0$ is the standard Frobenius endomorphism to $q$. 
As the isomorphism type of $\calQ$ depends not on $q$, the group $\ovT_1 \cap (\ovT_2\cdots \ovT_j)$ is finite and fixed by $F_0$. Using the assumption $\calQ\cong \calZ$ one obtains $\calZ= \omega^{-1}(\calQ)= \ovT_1 \cap (\ovT_2\cdots \ovT_j)$.
\end{proof}

If $\ovR$ is a root system of type $\tA_{l-1}$ or $\tC_l$ further conditions from the previous section can be verified.
\begin{lem}\label{lemAC}
Assume $\ovR$ to be a root system of type $\tA_{l-1}$ or $\tC_l$ and assume $p_r^2\in \Zent(\wK_r\wK_{r+1})$. Then
 for $r,r'\in J$ with $r\neq r'$ the equation $[\wK_r,\wK_{r'} ]=1$ holds and $[p_r,kk^{p_r}]=1$ for every $k \in \wK_r$.
\end{lem}
\begin{proof}
If $\ovR$ is a root system of type $\tA_{l-1}$ or $\tC_l$, then $\al\pm \beta \notin \ovR$ for $\al \in \ovR_r$ and $\beta \in \ovR_{r'}$ with $r\neq r'$. For example  $\ovR_r=\Set{2e_i, e_i-e_{i'}|i,i'\in \calB_r\text{, } i\neq i'}$ if \ovR is of type $\tC_{l}$. The roots from $\ovR$ can be seen in \ref{ass_classical_roots} and one observes $\al\pm \beta \notin \ovR$. According to the Steinberg relations this implies $[\wK_r, \wK_{r'}]=1$. 

For determining $[p_r,kk^{p_r}]$ with $r\in J'$ and $k \in \wK_r$ one calculates 
$\left(kk^{p_r}\right)^{p_r}= k^{p_r} k^{p_r^2}=k^{p_r} k= k k^{p_r}$, using $[\wK_r,\wK_{r+1}]=1$ and $p_r^2\in \Zent(\wK_{r}\wK_{r+1})$.
\end{proof}

\section{Application \texorpdfstring{to $\ovG^F= \tC_{l,\SC}(q)$}{to groups of type C}} \label{applC}
This section proves Theorem \ref{Theo1} in the situation where $\ovG$ has a root system of type $\tC_l$, by applying Proposition \ref{mFreg}. Hence $\ovG$ can be assumed to have a root system of type $\tC_l$ and $F$ to be a standard Frobenius endomorphism. The regular numbers of $(\ovG,F)$ are the divisors of $2l$ by Table \ref{TabregZahl}.

Thereby we use the calculations from Section \ref{sec_towards_application} after choosing a root system \ovR and a Frobenius endomorphism $\oF: \ovovG\rightarrow \ovovG$ satisfying Assumption \ref{ass_ovR_oF}. Furthermore we have to fix a Sylow $d$-twist $v\Gamma$ satisfying Assumption \ref{ass_Sylowtwist}.

\begin{ass} \label{assC}
In this section we assume $R$ to be the root system of type $\tC_l$ from \ref{ass_classical_roots}. Hence $\RF=\Lset{\al_1, \ldots, \al_l}$ with $\al_1:=2e_1$ and $\al_i:=e_{i}-e_{i-1}$ ($i\geq 2$) forms a system of simple roots of $R$.
With $\ovR:=R$ and $\oF:=F$ Assumption \ref{ass_ovR_oF} is satisfied and we may use the groups and maps introduced there.
Furthermore we assume $d\teilt 2l$.
\end{ass}

Using the morphism $f$ from \ref{ass_classical_roots} we determine a Sylow $d$-twist $v$ of $(\ovG, F)$ satisfying Assumption \ref{ass_Sylowtwist}.

\begin{lem}\label{DefvC}
Let $V$ be the extended Weyl group of $\ovG$ and $n_i:=n_{\al_i}(1)$ ($1\leq i \leq l$). Let $v_0:=n_1\cdots n_l\in V$ and $v:=v_0^{\frac{2l}{d}}$. Then $v$ is a Sylow $d$-twist of $(\ovG,F)$ satisfying Assumption \ref{ass_Sylowtwist}.
\end{lem}

\begin{proof}
We use the notion from the Definition \ref{def_good_roots} and examine the element $\bw:=\bb_1\cdots \bb_l$, which satisfies $\tau(\bw)=v$: The element $\bw^l$ has length $\leq l^2$ and $\rho\circ\tau(\bw^l)=w_0$ for the longest element $w_0$ in $W$. Hence $\bw$ is a good $2l$th root of $\bw_0^2$ in the associated braid group $\bB$. According to Remark \ref{remsylowb} the element $v= \tau(\bw^{\frac {2l} d})$ is a Sylow $d$-twist of $(\ovG,F)$. 

As $w_0$ is central in this Coxeter group $W$, the corresponding element $\bw_0$ is central in the braid group $\bB$ by \cite[4.4.1]{GeckCox}. Applying the epimorphism $\tau:\bB\rightarrow V$ gives that $\bww$, the canonical representative of $w_0$ in $\ovG$ from \ref{defgoodroot}, is central in $V$. Because of $v_0^l=\bww$ this proves $v^{\frac l j}\in \Zent(V)$. 

The further properties of $v$ follow from straightforward calculations in $V$ and $f(W)$, respectively.\end{proof}

We determine groups satisfying Assumption \ref{generalassumptionsa}.

\begin{lem}\label{lem_generalass_C}
Assume the groups $\wN$, $\wK$, $\wK_1$, $\wT$ and elements $\Set{p_r|r \in J'}$ be associated to $v$ as in Lemma \ref{lem1} with $\calZ=\Lset{1_\ovG}$. Then Assumption \ref{generalassumptionsa} is fulfilled.
\end{lem}

\begin{proof}
According to the above $v$ satisfies Assumption \ref{ass_Sylowtwist}. Hence we can apply Lemma \ref{lem1} if we additionally ensure the equations $\ovT=\ovT_{1}. \cdots .\ovT_{j}$ and $\calZ=\ovT_1\cap \prod_{1\neq r \in J}\ovT_{r}$, where 
$\ovT_r$ is defined as in Lemma \ref{lem1}.

According to Lemma \ref{lem3_lattices} both equations can be checked by calculations in the coroot lattices: straightforward calculations with lattices prove $2\al^\vee \in \oplus_{r\in J}\ZZ \ovR_r^{\vee}$ for all $\al\in \ovR$. 
Similarly the quotient \[\calQ:=\nicefrac{\spann< \ZZ \overline R_1^\vee , (q-1) \ZZ\overline R^\vee> \cap \Spann< \ZZ \overline R_{r}^\vee , (q-1) \ZZ \overline R^\vee | 1\neq r\in J>}{(q-1)\ZZ \overline R^\vee}\] can be calculated and is trivial. 

By Lemma \ref{lem3_lattices} this implies $\ovT=\ovT_{1} \cdots \ovT_{j}$ and $\calZ=\ovT_1\cap \prod_{1\neq r \in J}\ovT_{r}$. Hence Assumption \ref{generalassumptionsa} is satisfied.
\end{proof}

For applying Proposition \ref{mFreg} we also have to check the hypotheses.
\begin{lem}\label{lem_hyp_C}
\begin{enumerate}
\item Hypothesis \ref{hyp_inkommut} is satisfied with $Z=\calZ=\Lset{1_\ovG}$.
\item The elements $\Set{p_r|r\in J'}$ fulfil the hypotheses \ref{hypo_pr1} and \ref{hyp_pr2}.
\end{enumerate}
\end{lem}
\begin{proof}As $\calQ$ is trivial we have $\ovT_1\cap \prod_{1\neq r \in J }\ovT_r =\Lset{1_\ovG}$. This implies $Z=\Lset{1_\ovG}$. According to Lemma \ref{lemAC} the group $[\wK_r, \wK_{r'}]$ ($r\neq r' \in J$ ) is trivial. Hence Hypothesis \ref{hyp_inkommut} holds.

The first part of Hypothesis \ref{hypo_pr1} claims: for every $I\subseteq J'$ and $S_I:=\Spann<p_r|r \in I>$ the equation $S_I\cap \wT= \Spann<p_r^{2}| r \in I> $ holds. This is true by Lemma \ref{lem2a}.

The Steinberg presentation implies $p_r^2=\left(\prod_{i\in \calB_r} h_{e_i}(-1)\right)(\prod_{i\in \calB_{r+1}} h_{e_i}(-1))$ and $p_r^2=t (t^{p_r})^{-1}$ for $t:=\left(\prod_{i\in \calB_r} h_{e_i}(-1)\right)$. Hence the remaining part of Hypothesis \ref{hypo_pr1} is also true.

Now we concentrate on Hypothesis \ref{hyp_pr2}: using the Steinberg relations one obtains $t\in \Zent(\wK_r)$, which implies $p_r^2 \in \Zent(\wK_r\wK_{r+1})$. By Lemma \ref{lemAC} this gives $[p_r,kk^{p_r}]=1$ for every $k\in \wK_r$. This is the first part of Hypothesis \ref{hypo_pr1}. By Lemma \ref{lem2b} the second part of Hypothesis \ref{hyp_pr2} also holds, namely $S_I^s \wT= S_{I'} \wT$ for $s\in S$ and $I,I'\subseteq J'$ implies $S_I^s=S_{I'}$.\end{proof}

We are now ready to apply Proposition \ref{mFreg} and prove Theorem \ref{Theo1} for $\ovG^F=\tC_{l,\SC}(q)$.

\begin{lem}\label{Theo1C}
Theorem \ref{Theo1} holds whenever $\ovG$ has a root system of type $\tC_l$.
\end{lem}
\begin{proof}
According to the two previous lemmas the necessary conditions are satisfied. We use the trivial character of $\wN$ as $\delta$, which obviously satisfies Assumption \ref{notationdelta}.
By Proposition \ref{mFreg} maximal extensibility holds with respect to $\wT\lhd \wN$. By Lemma \ref{lem1} 
the groups $\wT$ and $\wN$ coincide with $\ovT^{vF}$ and $\ovN^{vF}$ respectively. 
According to Lemma \ref{Konstruktion_reg} the group $\ovT^{vF}$ is a Sylow $d$-Levi subgroup of $(\ovG,vF)$ and $\ovN^{vF}$ the associated Sylow $d$-normaliser. This shows the statement by Remark \ref{conj}.
\end{proof}

\section{Application \texorpdfstring{to $\ovG^F= \tA_{l-1,\SC}(q)$}{to special linear groups}}\label{applA}
In this section we prove Theorem \ref{Theo1} in the case where $\ovG^F$ is the special linear group over $\FF_q$, which we denote by $\tA_{l-1,\SC}(q)$. As in the previous section we apply Proposition \ref{mFreg} using the groups introduced in Section \ref{sec_mFreg}. The regular numbers of $(\ovG, F)$ are the divisors of $l$ and those of $l-1$, by Table \ref{TabregZahl}. These two cases are dealt separately. 

\begin{ass} \label{assA}
We assume $R$ to be the root system of type $\tA_{l-1}$ from Assumption \ref{ass_classical_roots} and associate to $R$ groups and morphisms as in Section \ref{sec_not}.

Let $\ovR$ be the root system of type $\tC_{l}$ from Assumption \ref{ass_classical_roots} with the set $\overline \RF=\Lset{\al'_0, \al_1,\ldots, \al_{l-1}}$ as a system of simple roots, where $\al_0':= 2e_1$ and $\al_1=e_2-e_1, \ldots , \al_{l-1}=e_l-e_{l+1}$.
Further let $\overline F:\overline \ovG\rightarrow \overline \ovG$ be the standard Frobenius endomorphism. Then $\ovR$ and $\oF$ satisfy Assumption \ref{ass_ovR_oF}. The set $\Lset{\al_1,\ldots, \al_{l-1}}$ is then a system of simple roots in $R$.
Let $d$ be a divisor of $l$.
\end{ass}

We choose the following good Sylow $d$-twist and prove that it satisfies Assumption \ref{ass_Sylowtwist}.

\begin{lem}\label{DefvA}
Let $V$ be the extended Weyl group of $\ovG$, $n_1, \ldots , n_{l-1}$ the generating elements of $V$ mentioned in \ref{EinfV}, $v_{0}:=n_1 n_2 \cdots n_{\floor{\frac{{l-1}}2}}n_{l-1}n_{l-2}\cdots n_{\floor{\frac{l-1}2}+1} \in V$ 
 and $v:=v_0^{\frac l d}$. Then $v$ satisfies Assumption \ref{ass_Sylowtwist}.

\end{lem}
\begin{proof} The proof is as in Lemma \ref{DefvC}. One shows that $v$ is the image of a good $d$th root of $\bw_0^2$ in the braid group and then applies Remark \ref{remsylowb}. From $v_0^l=\tau(\bw_0^2)$ and $\bw_0^2\in \Zent(\bB)$ one obtains $v_0^l\in \Zent(V)$. The remaining statements can be verified by calculations in $W\cong \Sym_l$.
\end{proof}

For regular numbers $d$ of $(\ovG,F)$ with $d\nmid l$ there exists no Sylow $d$-twist fulfilling Assumption \ref{ass_Sylowtwist} as the associated orbits are of different length. Hence we deal with these regular numbers later.

\begin{lem} \label{lem_generalass_A}
\begin{enumerate}
	\item Assume the groups $\wN$, $\wK$, $\wK_1$, $\wT$ and elements $\Set{p_r|r \in J'}$ be associated to $v$ as in Lemma \ref{lem1} with $\calZ=\Lset{1_\ovG}$. Then Assumption \ref{generalassumptionsa} is fulfilled. 
	\item The group $Z=\calZ=\Lset{1_\ovG}$ satisfies Hypothesis \ref{hyp_inkommut}.
	\item The elements $\Set{p_r|r\in J'}$ fulfil the hypotheses \ref{hypo_pr1} and \ref{hyp_pr2}.
\end{enumerate}
\end{lem}

\begin{proof} Part (a) can be proven as Lemma \ref{lem_generalass_C}. The arguments of Lemma \ref{lem_hyp_C} imply the remaining statements.\end{proof}

Here we choose $\delta$ to be a non-trivial character, as $\wN$ is not isomorphic to a Sylow $d$-normaliser of $(\ovG, F)$.

\begin{lem}\label{defdeltaA} 
There exists a linear character $\delta\in \Irr(\wN)$ satisfying $\ker(\delta)\cap \wT= \ovT^{vF}$, $\ker(\delta)= \ovN^{vF}$ and Assumption \ref{notationdelta}.
\end{lem}
\begin{proof} A detailed proof of this fact can be found in \cite[8.1.7]{Spaeth}. A key argument is that the group $\nicefrac \ovovT \ovT$ is cyclic. Hence a linear character on $\ovovT$ with kernel $\ovT $ exists. This character can be extended to $\ovN \ovovT$. Because of $\wK_r=\spann<\ovN\cap \wK_r, \wT>$ and $p_r\in \ovN^{vF} $ ($r\in J'$) the Assumption \ref{notationdelta} is satisfied.
\end{proof} 

Now we may apply Proposition \ref{mFreg} and verify Theorem \ref{Theo1} for $d\teilt l$.

\begin{lem} \label{Adl} Let $\ovG = \tA_{l-1,\SC} (\ovF_q )$ ($l\geq 2$), $F:\ovG\rightarrow \ovG$ the standard Frobenius endomorphism, $d\teilt l$, $T$ a Sylow $d$-Levi subgroup of $(\ovG,F)$ and $N$ the associated Sylow $d$-normaliser. Then maximal extensibility holds with respect to $T\lhd N$. \end{lem}
\begin{proof}
The assumptions of Proposition \ref{mFreg} are satisfied according to the above. This implies that maximal extensibility holds with respect to $\ker(\delta)\cap \wT \lhd \ker(\delta)$, specifically $\ovT^{vF}\lhd \ovN^{vF}$ by the definition of $\delta$.

According to Lemma \ref{Konstruktion_reg} the group $\ovT^{vF}$ is a Sylow $d$-Levi subgroup and $\ovN^{vF}$ the associated Sylow $d$-normaliser. By Remark \ref{conj} this implies the statement.\end{proof}

Now we concentrate on the remaining regular numbers.
\begin{lem} \label{Adl-1}
Let $\ovG = \tA_{l-1,\SC} (\ovF_q )$ ($l\geq 2$), $F:\ovG\rightarrow \ovG$ the standard Frobenius endomorphism, $1\neq d\teilt (l-1)$, $T$ a Sylow $d$-Levi subgroup of $(\ovG,F)$ and $N$ the associated Sylow $d$-normaliser. Then maximal extensibility holds with respect to $T\lhd N$.\end{lem}

\begin{proof} The condition $1\neq d\teilt (l-1)$ implies $l\geq 3$. As in Assumption \ref{assA} we assume $\ovR$ to be the root system of type $\tC_l$ and $R\subseteq \ovR$ the root system of type $\tA_{l-1}$ from \ref{ass_classical_roots}. We use the groups and maps associated to $\ovR$ and $R$ in Assumption \ref{ass_ovR_oF}.

Let $\ovR'\subset \ovR$ be the root system of type $\tC_{l-1}$ and $R'\subset R$ the root system of type $\tA_{l-2}$ from Assumption \ref{ass_classical_roots}. Let $\ovG':=\Spann<\ovX_\al|\al \in R'>$. 

Because of $d\teilt (l-1)$ the integer $d$ is a regular number of $(\ovG',\restr F|{\ovG'})$. The Sylow $d$-torus of $(\ovG', \restr F|{\ovG'})$ is also a Sylow $d$-torus of $(\ovG, F)$, as the associated polynomial orders, which can be computed via $\betrag{\SL_l(q)}$, coincide. Hence the Sylow $d$-twist $v$ of $(\ovG',\restr F|{\ovG'})$ is a Sylow $d$-twist of $(\ovG,F)$. 

By Lemma \ref{Konstruktion_reg} the group $T:=\ovT^{vF}$ is a Sylow $d$-Levi subgroup of $(\ovG, vF)$ and $N:=\ovN^{vF}$ the associated Sylow $d$-normaliser. Let $T':=\ovovT_{\ovR'}^{vF}$ and $N':=(\ovN_{R'} \ovovT_{\ovR'})^{v\oF}$, where $\oF:\ovovG\rightarrow \ovovG$ is the standard Frobenius with $\restr \oF|{\ovG}=F$. Calculations with root lattices using \cite[3.2.3]{Car1} show that $T'$ and $T$ are isomorphic. The Steinberg relations show that this isomorphism can be extended to an isomorphism of $N'$ and $\ovN^{vF}$.

According to the above explanation the groups $N'$ and $T'$ satisfy the assumptions about $\wN$ and $\wT$ from Section \ref{sec_mFreg}. Hence maximal extensibility holds with respect to $T'\lhd N'$ according to Proposition \ref{mFreg}, where in the application $\delta$ is chosen to be the trivial character of $\wN$.
\end{proof}

Altogether the previous two lemmas have shown Theorem \ref{Theo1} for $\ovG^F=\tA_{l-1, \SC}(q)$.

\section{Application\texorpdfstring{ to $\ovG^F=\tw 2 \tA_{l-1,\SC}(q)$}{ to special unitary groups}} \label{appl2A}
The aim of this section is to verify Theorem \ref{Theo1} in the case where $\ovG^F=\tw 2 \tA_{l-1,\SC}(q)$, namely a special unitary group over a finite field. As in the previous two sections we use Proposition \ref{mFreg} for this purpose.

Like before we start by choosing a root system and a Frobenius endomorphism satisfying assumption \ref{ass_ovR_oF}. From Table \ref{TabregZahl} we know the regular numbers of $(\ovG,F)$. Like in the previous section we have two cases depending on the regular numbers.

\begin{ass}\label{ass2A}
Let $\overline R$, $\overline \RF$, $R$ and $\RF$ be defined as in Assumption \ref{assA} and assume the groups $\ovovN$, $\ovovT$, $\ovovX_\al$ ($\al \in \ovR$), $\ovW$ and $\ovrho$ be determined by $R$ and $\ovR$ as in Assumption \ref{ass_ovR_oF}, where $\ovG=\Spann< \ovovX_\al|\al \in R > $ by Remark \ref{einbett}.

Let $\gamma$ be the non-trivial graph automorphism of $R$, which stabilises $\RF$, and $ F: \ovG \rightarrow \ovG$ the Frobenius endomorphism with $F=F_0\circ \calG$, where $F_0$ is a standard Frobenius endomorphism and $\Gamma$ the associated graph automorphism with $x_\alpha(t)\mapsto x_{\overline \alpha}(t)$ ($\pm \alpha \in \RF$, $\overline \al:= \gamma(\al)$). Assume $d\teilt l$ for even $l$, and $2\teilt d$ and $d\teilt 2l$, otherwise.
\end{ass}

We construct a Frobenius endomorphism $\oF:\overline \ovG\rightarrow \ovG$ with $\restr \oF\,\,| {\ovG}=F$ for Assumption \ref{ass_ovR_oF}.
\begin{lem} \label{1113}
Let $\bww \in V$ be defined as in Definition \ref{defgoodroot} and $\overline \bww$ the analogous element of $\overline \ovG$. 
Then $n:=\bww\overline \bww$ satisfies $\Gamma(x)= x^n$ for all $x\in \ovG$.
\end{lem}
\begin{proof}
Because of $\Spann<\ovX_\al|\pm \al \in \RF>= \ovG$ it suffices to prove the statement for all $x=x_ \alpha(t)$ ($\pm \alpha \in \RF$, $t\in \ovF_q$). According to the Steinberg relations we have $x_\alpha(t)^n= x_{\overline \alpha}(i_\alpha t)$ for some $i_\alpha \in \Lset{\pm 1}$. Let $\eta_{\beta, \al} \in \Lset{\pm 1}$ such that $ x_\al(t) ^{n_\beta(1)}=x_{w_{\beta}(\al)}(\eta_{\beta,\al}t)$ for every $t\in \ovF_p$. From $n= n_{\beta_{1}}(1)\cdots n_{\beta_{l'}}(1)$ for some roots $\beta_i$ we obtain $ i_{\alpha}= \eta_{\beta_{1},\alpha} \eta_{\beta_{2},w_{\beta_{1}}(\alpha)} \cdots $. By \cite[Proposition 6.4.3]{Car1} we know $ \eta_{\al,\beta}=\eta_{\al,-\beta} $ for all $ \al,\beta \in R$. This proves $i_\alpha= i_{-\alpha}$ for all $\al \in \RF$ and $ n_\alpha(t)^n= x_{\overline \alpha}(i_\alpha t) x_{-\overline\alpha}(i_\alpha t^{-1}) x_{\overline\alpha}(i_ \alpha t )= n_{\overline \alpha} (i_\alpha t)$. Using \cite[4.1.9]{GeckCox} and calculations in $\overline W$ we obtain $ n_i(1)^n= n_{l-i}(1)$ for all $ 1\leq i \leq l-1$.
This proves $i_\alpha=1$ and $\Gamma(x_\al(t))= x_\al(t)^n$ for all $\al \in \RF$.
\end{proof}

The Frobenius endomorphism $\oF:=n\oF_0:\overline \ovG\rightarrow \overline \ovG$ with $n$ from Lemma \ref{1113} satisfies Assumption \ref{ass_ovR_oF}, specifically $\restr \oF|{\ovG}=F$, where $\oF_0:\overline \ovG\rightarrow \overline \ovG$ is the standard Frobenius endomorphism of $\overline \ovG$.
\begin{lem}\label{Defv2A}
Let $n_1, \ldots, n_{l-1}$ be as in \ref{EinfV}. For $2\teilt l $ let \[v_0:=n_1 n_2 \cdots n_{\floor{\frac{{l-1}}2}} n_{l-1}n_{l-2}\cdots n_{\floor{\frac{l-1}2}+1}\] and $v\in V$ such that $v\Gamma =v_0^\je \Gamma$ with $\je =\frac l d$. For $2\nmid\ell$ let $v_0:=n_1 n_2 \cdots n_{\floor{\frac{{l}}2}}$ and $v\in V$ such that $v\Gamma:=(v_0 \Gamma)^\je$ with $\je:=\frac{2l}d$. Then $v\Gamma$ is a Sylow $d$-twist of $(\ovG,F)$ satisfying Assumption \ref{ass_Sylowtwist}.
\end{lem}
\begin{proof} The proof in the same spirit as the one of Lemma \ref{DefvA}.
\end{proof}

With this knowledge we introduce groups satisfying Assumption \ref{generalassumptionsa}.

\begin{lem}
\begin{enumerate}
\item \label{lem_generalass_2A}
Assume the groups $\wN$, $\wK$, $\wK_1$, $\wT$ and elements $\Set{p_r|r \in J'}$ be determined by $v$ as in Lemma \ref{lem1} with $\calZ:=\Lset{1_\ovG}$. Then Assumption \ref{generalassumptionsa} is fulfilled.
\item Hypothesis \ref{hyp_inkommut} is satisfied with $Z=1$.
\item The elements $\Set{p_r|r\in J'}$ fulfil the hypotheses \ref{hypo_pr1} and \ref{hyp_pr2}.
\end{enumerate}
\end{lem}
\begin{proof} The part (a) can be proven as Lemma \ref{lem_generalass_C}. The proof of (b) and (c) can be seen in the one of Lemma \ref{lem_hyp_C}.\end{proof}

As in the previous section we need to construct a non-trivial character $\delta \in \Irr(\wN)$ such that $\ker(\delta)$ is a Sylow $d$-normaliser.

\begin{lem}\label{defdelta2A}
There exists a linear character $\delta \in \Irr(\wN)$ with $ \ker(\delta)\cap \overline \ovT= \ovT^{vF}$ and $ \ker(\delta)= \ovN^{vF}$ satisfying Assumption \ref{notationdelta}.
\end{lem}
\begin{proof} 
By Lemma \ref{defdeltaA} there exists a linear character $\delta'\in \Irr((\ovN\ovovT)^{\oF_0^{2d}})$ with $\ker(\delta) = \ovN^{\oF_0^{2d}}$, where $\oF_0$ is a standard Frobenius endomorphism to $q$. As $(v\oF)^{2d}$ and $\oF_0^{2d}$ induce the same automorphism on $\ovN\ovovT$ we have $\wN\leq (\ovN \ovovT)^{\oF_0^{2d}}$ and the character $\delta:=\restr \delta'|{\wN}$ is well-defined. By definition the equations $\ker(\delta)=\ovN^{vF}$ and $ \ker(\delta)\cap \ovovT= \ovT^{vF}$ hold and Assumption \ref{notationdelta} is satisfied.
\end{proof}

Now we can apply Proposition \ref{mFreg} again.

\begin{lem}\label{Hauptfall}
Assume $\ovG$, $F$ and $d$ be given as in \ref{ass2A}. Then Theorem \ref{Theo1} holds, if $\ovS$ is a Sylow $d$-torus of $(\ovG,F)$.\end{lem}
\begin{proof}
As in the proof of Lemma \ref{Adl} one applies Proposition \ref{mFreg} and obtains the result.
\end{proof}

We are left with the remaining regular numbers. In the following we keep the notion from \ref{ass2A}, without the assumptions on $d$. We construct a Frobenius endomorphism $\widetilde F: \ovG\rightarrow \ovG$, whose restriction to $\Spann<\ovX_\al| \al \in R'>$ give endomorphisms of these groups, where $R'$ is any root subsystem of $R$.

\begin{lem}[{\cite[9.1.12]{Spaeth}}]\label{neuerFrob2A}
Let $\bww\in V$ be the canonical representative of $w_0 \in W$ from Definition \ref{defgoodroot}.
\begin{enumerate}
\item The automorphism $\widetilde \calG:=\bww\calG$ acts on $\ovG$ by $x_{\alpha}(t) \mapsto x_{-\alpha}(-t)$ ($ \pm \alpha \in \RF$).
\item \label{Feig}
For $1<l'\leq l$ let $R'$ be the parabolic root system of $R$ with simple roots $\Lset{\al_1, \ldots , \al_{l'-1}}$, $\ovG':=\Spann<\ovX_\al| \al \in R'>$ and $\widetilde \calG': \ovG'\rightarrow \ovG'$ the automorphism defined like in (a) but for $\ovG'$. Then $\widetilde \calG'= \restr \widetilde \calG|{\ovG'}$.
\end{enumerate}
\end{lem}
\begin{proof} The first part is a consequence of longish calculations with the Steinberg relations, similar to the ones proving Lemma \ref{1113}. The second follows easily from (a).\end{proof}

Using this automorphism we construct a Sylow $d$-twist.
\begin{lem} \label{v2Ai}
Let $d$ be a regular number of $(\ovG,F)$, which does not satisfy Assumption \ref{ass2A}, and $l':=l-1$. Let $R'$, $\ovG'$ and $\widetilde \calG': \ovG' \rightarrow \ovG'$ be associated to $l'$ as in Lemma \ref{Feig} and $v\widetilde \calG'$ a Sylow $d$-twist of $(\ovG', \widetilde F')$, where $\widetilde F':=\calG' \circ(\restr F_0|{\ovG'})$.
\begin{enumerate}
\item Then $d$ is a regular number of $(\ovG', \widetilde F')$ satisfying the assumptions from \ref{ass2A}.
\item The automorphism $v\widetilde \calG$ is a Sylow $d$-twist of $(\ovG, \widetilde F)$ with $\widetilde F:= \calG \circ F_0$.
\end{enumerate}
\end{lem}
\begin{proof}
The first part follows easily from Table \ref{TabregZahl}. For (b) one calculates the polynomial order of Sylow $d$-tori of $(\ovG',\widetilde F')$ and $(\ovG,\widetilde F)$ via the formulas for $\betrag{\ovG^F}$ in \cite[2.9]{Car2}. As these coincide the result follows. \end{proof}

This implies the following lemma.
\begin{lem}\label{Theo12A}
Theorem \ref{Theo1} holds, if $\ovG$ has a root system of type $\tA_{l-1}$ and $F$ is induced from a graph automorphism. 
\end{lem}
\begin{proof}This statement is proven for some regular numbers by Lemma \ref{Hauptfall}. For the remaining cases we know that $d$ is a regular number of $(\ovG', \widetilde F')$ satisfying the assumptions from \ref{ass2A}. In this situation we compare $\ovN^{v\widetilde F}$ with $(\ovovT_{\overline R'} \ovN_{R'})^{v\widetilde \oF'}$, where $\widetilde \oF':= \bww' \oF$, the element $\bww'$ is the canonical representative of the longest element in $W_{R'}$ and $\ovR'$ the root subsystem of type $\tB_{l-1}$ with simple roots $\Lset{\al_0',\al_1, \ldots , \al_{l-2}}$ from \ref{assA}.

Like in the proof of Lemma \ref{Adl-1} there exists an isomorphism between $\ovN^{v\widetilde F}$ and $(\ovovT_{\overline R'} \ovN_{R'})^{v\widetilde \oF'}$, which maps $\ovT^{v\widetilde F}$ onto $\ovovT_{\overline R'}^{v\widetilde \oF'}$. These subgroups of $\ovovG'$ occurred already as $\wN$ and $\wT$ in the proof of Lemma \ref{Hauptfall}. From there it is known that maximal extensibility holds with respect to $\wT\lhd \wN$. This implies the statement. 
\end{proof}

\section{Application \texorpdfstring{for $\ovG^F=\tB_{l,\SC}(q)$}{for groups of type B}} \label{applB}
In this section we concentrate on the case where the root system of $\ovG$ is of type $\tB_l$ and $F$ is a standard Frobenius endomorphism. Again we prove Theorem \ref{Theo1} by applying the Proposition \ref{mFreg}. 

Like before we use results from Section \ref{sec_towards_application} after choosing a root system \ovR and a Frobenius endomorphism $\oF: \ovovG\rightarrow \ovovG$ satisfying Assumption \ref{ass_ovR_oF}. Furthermore we have to fix a Sylow $d$-twist $v\Gamma$ satisfying Assumption \ref{ass_Sylowtwist}.

\begin{ass}\label{assB}
In this section we assume $R$ to be the root system of type $\tB_l$ from \ref{ass_classical_roots}. Hence $\RF=\Lset{\al_1, \ldots, \al_l}$ with $\al_1:=e_1$ and $\al_i:=e_{i}-e_{i-1}$ ($i\geq 2$) forms a system of simple roots of $R$.
With $\ovR:=R$ and $\oF:=F$ Assumption \ref{ass_ovR_oF} is satisfied and we use the groups and maps introduced there.
Assume $d$ to be a regular number of $(\ovG,F)$, hence $d\teilt 2l$ by Table \ref{TabregZahl}.
\end{ass}

Using the morphism $f$ from \ref{ass_classical_roots} we determine a Sylow $d$-twist $v$ of $(\ovG, F)$ satisfying Assumption \ref{ass_Sylowtwist}.

\begin{lem}\label{DefvB}
Let $V$ be the extended Weyl group of $\ovG$ and $n_i:=n_{\al_i}(1)$ ($1\leq i \leq l$). Let $v_0:=n_1\cdots n_l\in V$ and $v:=v_0^{\frac{2l}{d}}$. Then $v$ is a Sylow $d$-twist of $(\ovG,F)$ satisfying Assumption \ref{ass_Sylowtwist}.
\end{lem}

\begin{proof} By the proof of Lemma \ref{DefvC} the element $v$ is the image of a good $d$th root of $\bbw^2$ in the associated braid group, as the braid group of type $\tB_l$ and $\tC_l$ coincide. Analogously the Weyl groups coincide and Lemma \ref{DefvC} also implies that the remaining conditions of Assumption \ref{ass_Sylowtwist} are satisfied.\end{proof}

We determine groups satisfying Assumption \ref{generalassumptionsa}.

\begin{lem}\label{lem_generalass_B}
Assume the groups $\wN$, $\wK$, $\wK_1$, $\wT$ and elements $\Set{p_r|r \in J'}$ be associated to $v$ as in Lemma \ref{lem1} with $\calZ=\spann<h_{\al_1}(-1)>$. Then Assumption \ref{generalassumptionsa} holds.
\end{lem}
\begin{proof} Because of Lemma \ref{DefvB} we can apply Lemma \ref{lem1} if $\ovT=\ovT_{1}. \cdots .\ovT_{j}$ and $\calZ=\ovT_1\cap \prod_{1\neq r \in J}\ovT_{r}$, where $\ovT_r$ is defined as in Lemma \ref{lem1}
Like before both equations can be tested by calculations in the coroot lattices:
Again one checks $2\al^\vee \in \oplus_{r\in J}\ZZ \ovR_r^{\vee}$ for all $\al\in \ovR$. 
Similarly the quotient group \break $\calQ:=\nicefrac{\spann< \ZZ \overline R_1^\vee , (q-1) \ZZ\overline R^\vee> \cap \Spann< \ZZ \overline R_{r}^\vee , (q-1) \ZZ \overline R^\vee | 1\neq r\in J>}{(q-1)\ZZ \overline R^\vee}$ can be calculated and 
one obtains $\calQ=\nicefrac{\Spann<(q-1)e_i|1\leq i \leq l>} {(q-1)\ZZ \overline R^\vee}$. Hence the group has order $(2,q-1)$.

By Lemma \ref{lem3_lattices} this implies $\ovT=\ovT_{1} \cdots \ovT_{j}$ and $\calZ=\ovT_1\cap \prod_{1\neq r \in J}\ovT_{r}$. Hence Assumption \ref{generalassumptionsa} is satisfied.
\end{proof}

For applying Proposition \ref{mFreg} we also have to check also the hypotheses, which is a bit more complicated in this case because of $Z\neq 1$.

\begin{lem}\label{lem_hyp_B}
\begin{enumerate}
\item The group $Z=\calZ$ satisfies Hypothesis \ref{hyp_inkommut}.
\item The elements $\Set{p_r|r\in J'}$ fulfil the hypotheses \ref{hypo_pr1} and \ref{hyp_pr2}.
\end{enumerate}
\end{lem}
\begin{proof} It is well-known that $Z$ is the centre of $\ovG$. Hence it suffices for (a) to verify $ [\I_{\wK_r}(\lambda), \wK_{r'}]\leq \ker(\lambda)\cap Z $ for all $ r\neq r'$ and $\lambda\in \Irr(\wT_r)$.

First we begin to calculate $[\wK_r, \wK_{r'}]$ for $r\neq r'$. As every long root $\beta\in R_r$ satisfies $\beta\perp \al$ and $\beta\pm \al \notin R_{r'}$, the equation $[\ovN_{R_r^{\tD}},\ovK_{r'}]=1 $ holds, where $R_r^{\tD}$ is the root system of long roots of $R_r$. Longish calculations with the Steinberg relations also show $[n_{e_r}(1),n_{e_{r'}}(1)]=h_{e_1}(-1)$, see \cite[Lemma 10.1.5]{Spaeth}. This implies $[\wK_r, \wK_{r'}]\leq Z$.

Let $\la \in \Irr(\wT_r)$. If $Z$ lies in $\ker(\la)$, the above implies $[\I_{\wK_r}(\lambda), \wK_{r'}]\leq Z=\ker(\la)\cap Z $. Otherwise $2\nteilt q$, and the character $\lambda\in \Irr(\wT_r)$ satisfies $\lambda(h_{e_1}(-1))=-1$ and ${\I_{\wK_r}(\lambda)\leq \Cent_{\ovN_r}(\ovN_{r+1})}$, as $t^k= h_{e_1}(-1) t$ and $\la(t)\neq \la(t^k)$ for $k\in \wK_r\setminus \ovN_{R_r^{\tD}}$ and $ t:=\prod_{i\in \calB_r} h_{e_i}(\zeta)$, where $\zeta$ is a primitive fourth root of unity. The element $t$ lies in $\wT_r$, as $t^{vF}\in \Lset{t, t h_{e_r}(-1)}$. This proves $ [\I_{\wK_r}(\lambda), \wK_{r'}]\leq \ker(\lambda)\cap Z $.

The first part of Hypothesis \ref{hypo_pr1} states: For every $I\subseteq J'$ and $S_I:=\Spann<p_r|r \in I>$ the equation $S_I\cap \wT= \Spann<p_r^{2}| r \in I> $ holds. This is known from Lemma \ref{lem2a}.

The Steinberg presentation implies $p_r^2=\left(\prod_{i\in \calB_r} h_{e_i}(\zeta)\right)(\prod_{i\in \calB_{r+1}} h_{e_i}(\zeta^{-1}))$, where $\zeta$ is a primitive $(2,q-1)^2$th root of unity. This implies $p_r^2=t (t^{p_r})^{-1}$ for $t:=\prod_{i\in \calB_r} h_{e_i}(\zeta)$. Hence the remaining part of Hypothesis \ref{hypo_pr1} is also true.

Now we concentrate on Hypothesis \ref{hyp_pr2}: for $k\in \wK_r$ the Steinberg relations imply $t^k=\prod_{i\in \calB_r} h_{e_i}(\zeta)=t$, if $\rho(k)\in W_{R_r^\tD}$ or $2\teilt q$, and $t^k=h_{e_r}(-1) t$, otherwise. For $r'\in J'$ this implies \[(k k^{p_r})^{p_r}=k^{p_r} k^{p_r^2}= k^{p_r} k^t=k^{p_r} k h_{e_r}(-1)= k k^{p_r},\] if $k\in \wK\setminus \Cent_{\wK_r}(\wK_{r+1})$. Analogous calculations for $k\in \Cent_{\wK_r}(\wK_{r+1})$ imply $[p_r,kk^{p_r}]=1$ for every $k\in \wK_r$. This is the first part of Hypothesis \ref{hypo_pr1}. By Lemma \ref{lem2b} the second part of Hypothesis \ref{hyp_pr2} also holds.\end{proof}

We are now ready to apply Proposition \ref{mFreg} and prove Theorem \ref{Theo1} for $\ovG^F=\tB_{l,\SC}(q)$.

\begin{lem}\label{Theo1B}
Theorem \ref{Theo1} holds whenever $\ovG$ has a root system of type $\tB_l$.
\end{lem}
\begin{proof}
According to the previous lemmas the necessary conditions and assumptions are satisfied. Let $\eta$ be a faithful linear character of $Z$. Then $\delta(n):=\eta(\Lang(n))$ defines a linear character of $\wN$, whose kernel is $\ovN^{vF}$. The character satisfies $\delta(S)=1$ as by definition $p_r\in \ovN^{vF}$. The theorem of Lang \cite[4.4.17]{Springer} implies $\rho(\ovK_r^{vF})=\rho(\wK_r)$. Hence $\delta$ also satisfies $\ker(\delta_r)\wT_r=\wK_r$.

By Proposition \ref{mFreg} maximal extensibility holds with respect to $\ovT^{vF}\lhd \ovN^{vF}$, where by Lemma \ref{Konstruktion_reg} the group $\ovT^{vF}$ is a Sylow $d$-Levi subgroup of $(\ovG,vF)$ and $\ovN^{vF}$ the associated Sylow $d$-normaliser.
\end{proof}

\section{Application \texorpdfstring{to $\ovG^F=\tD_{l,\SC}(q)$}{to groups of type D}} \label{applD}
In this section we prove Theorem \ref{Theo1} in the case where the root system of $\ovG$ is of type $\tD_l$ and $F$ is a standard Frobenius endomorphism.

In contrary to the previous cases, we cannot apply Proposition \ref{mFreg} anymore, as the relative Weyl groups in the situations are no wreath products in general. We treat $\ovG$ as a subgroup of a simply-connected simple group associated with root system of type $\tB_l$.

\begin{setting}\label{set51D}
Let $\overline R$ be a root system of type $\tB_l$ with the system of simple roots $\overline \RF=\Lset{\overline\al_1,\al_2, \ldots , \al_l}$ from Assumption \ref{assB} and $\ovovG$ the associated simply-connected simple algebraic group over $\ovF_q$, whose root subgroups are denoted by $\ovovX_\al$ as in \ref{basicsetting}. Let $R\subset \overline{R}$ be the root subsystem consisting of all long roots with $\RF=\Lset{\al_1,\al_2, \ldots , \al_l}$ as system of simple roots, where $\al_1=2\overline \al_1+\al_2$. By Remark \ref{einbett} the group $\ovG$, a simply-connected simple group associated to $R$, can be canonically identified with $\Spann<\ovX_\al|{\al \in R}>\leq \overline \ovG$ because of ${\ZZ R^{\vee}= \ZZ \overline R^\vee}$.

With the definitions from Section \ref{sec_not} we associate to $\ovR$ also the groups $\ovovN$, $\ovovT$ and $\ovW$, and additionally define $\ovrho:\ovovN \rightarrow \ovW$ to be the canonical epimorphism. Because of $\Spann<\ovovX_\al|\al \in R>=\ovG$ the inclusions $\ovT\leq \ovovT$, $\ovN \leq \ovovN$ and $W\leq \ovW$ hold, furthermore $\restr \ovrho|{\ovN}=\rho$.
Straightforward calculations also show $\ovT=\overline \ovT$, $\ovN= \ovovN\cap \ovG$ and $\ovN=\Set{x\in \ovovN| \overline \rho(x)\in W}$.

Let $\oF:\ovovG\rightarrow \ovovG$ and $F:\ovG\rightarrow \ovG$ be the standard Frobenius endomorphisms associated to $q$.
\end{setting}
Using Sylow $d$-twists of $(\ovovG,\oF)$ one obtains a Sylow $d$-twist of $(\ovG,F)$.
\begin{lem} \label{defvD1}
Let $d$ be a regular number of $(\ovG,F)$, $l'= \begin{cases} l & d\teilt l,\\
l-1 & d\nteilt l,\end{cases}$, $\ovR':=\ovR\cap\Spann<e_i|1\leq i\leq l'>$ with simple roots $\RF':=\Lset{\overline \al_1, \al_2,\ldots, \al_{l'}}$ and $\ovovG':=\Spann<\ovX_\al| \al \in \overline R' >$. For every Sylow $d$-twist $v'$ of $(\ovovG', \restr \oF\,|{\ovovG'})$ the element $v:= \begin{cases} v'& v' \in \ovG,\\
v'n_{e_l}(-1)& v' \notin \ovG\end{cases}$ is a Sylow $d$-twist of $(\ovG,F)$.
\end{lem}
\begin{proof}
By Remark \ref{einbett} the group $\ovovG'$ is the simply-connected simple algebraic group associated to $\overline R'$. 
Using the polynomial order associated to $(\ovG,F)$ and $(\ovovG',\restr \oF\,|{\ovovG'})$ known from \cite[2.9]{Car2} one verifies that the polynomial order of the Sylow $d$-tori of both groups coincide after some calculations. Furthermore by Table \ref{TabregZahl} we have $d \teilt l$ or $d\teilt 2(l-1)$, and hence $d\teilt 2l'$. By the same table $d$ is then a regular number of $(\ovovG',\restr \oF|{\ovovG'})$.

By the definition of Sylow $d$-twist and because of $\ZZ \overline R^\vee=\ZZ R^\vee$ the statement holds if $l'=l-1$, otherwise we have $d\teilt l$. Then according to calculations with $f(\rho(v))$  the permutation $\rho(v')$ maps $0$ or $\frac{2l}d$ positive numbers to negative ones and hence $\rho(v')\in W$. This implies $v'\in \ovN$ and $v'=v$ because of $\ovN=\Set{x\in \ovovN| \overline \rho(x)\in W}$.
\end{proof}

This proves Theorem \ref{Theo1} in the case where $l'=l$.

\begin{lem}
\label{Theo1DI}
If $l'=l$ in Lemma \ref{defvD1} maximal extensibility holds with respect to $\ovT^{vF}\lhd\ovN^{vF}$.
The relative inertia groups for $\la \in \Irr(\ovT^{vF})$ is $\rho(\I_{\ovovN^{vF}}(\la)) \cap W$.
\end{lem}
\begin{proof} We know that the element $v$ from Lemma \ref{defvD1} is a Sylow $d$-twist. By Lemma \ref{Konstruktion_reg} this implies that $T:=\ovT^{vF}$ is a Sylow $d$-Levi subgroup and $N=\ovN^{vF}$ the associated Sylow $d$-normaliser.

As $l'=l$ implies $v'=v$ according to the above proof we have $T= \ovovT^{v'\oF}=:\overline T$ and the inclusion $N\leq \ovovN^{v'F}=:\overline N$. By Lemma \ref{Theo1B} maximal extensibility holds with respect to $\overline T\lhd \overline N$. This implies the statement by \cite[Lemma 4.2]{Spaeth_Preprint}.

Because of $N=\overline N\cap \rho^{-1}(W)$ the equation $\I_{N}(\la)=\I_{\overline N}(\la)\cap N$ holds for every $\la\in \Irr(T)$.\end{proof}

In the case of $l'=l-1$ we cannot use results about $(\ovovG, \oF)$ in an analogous way, as every Sylow $d$-torus of $(\ovG, F)$ is either a non-regular Sylow $d$-torus of $(\ovovG,\oF)$ or is not a Sylow $d$-torus of $(\ovovG,\oF)$.
\begin{lem} \label{Theo1DII} 
If $l'=l-1$ in Lemma \ref{defvD1}, maximal extensibility holds with respect to $\ovT^{vF}\lhd\ovN^{vF}$.
\end{lem}
\begin{proof}
Let $\ovovN':=\ovovN_{\overline R'}$ and $\ovovT':=\ovovT_{\overline R'}$, $\overline N':=\ovovN'^{v'\oF}$ and $\overline T':=\ovovT'^{v'\oF}$. Then $N:=\ovN^{vF}$ and $\widetilde N:=\spann< \overline N'\cap \ovN, n n_{e_l}(1) >$ with $n\in \overline N'\setminus N$ satisfy $\widetilde N\cap \ovT=\overline T'$ and $N=\widetilde N T$. The equation $\widetilde N\cap \ovT= T'$ follows directly from the definition of the groups. 

For verifying $N=\widetilde N T$ we only have to verify $\rho(\widetilde N)=\rho(N)$ and $\widetilde N\leq N$. Using the Steinberg relations for the orthogonal roots one obtains after some calculations $\widetilde N\leq N$.

Furthermore the groups $\overline \rho( N)$ and $\overline\rho(\widetilde N)$ coincide according to calculations in $\overline W$ and $W$.

As $T'$ is a Sylow $d$-Levi subgroup and $N$ the associated Sylow $d$-normaliser, maximal extensibility holds with respect to $T'\lhd \overline N'$ by Lemma \ref{Theo1B}. Let $\la\in \Irr(\ovT^{vF})$, $\la'\in \IRR(T'|\la)$ and $\widetilde \la'$ a maximal extension of $\la$ in $\overline N'$. The character $\restr \widetilde \la'|{\I_{\overline N'\cap \widetilde N}(\la)}$ is invariant under $n\in \I_{\overline N'}(\la)\setminus \widetilde N$ and also under $nn_{e_l}(1)$, as both elements induce the same automorphism on $\overline N'\cap \widetilde N$. Because of $\I_{\widetilde N}(\la)=\spann<\I_{\overline N'\cap \widetilde N}(\la), nn_{e_l}(1)> $ the character $\la$ is maximally extendible in $\widetilde N$.
By Lemma 4.2 of \cite{Spaeth_Preprint} this implies maximal extensibility with respect to $T\lhd N$. According to Lemma \ref{Konstruktion_reg} this proves the statement.

The relative inertia group of $\la $ in $N$ can be isomorphic to the one of $\la'$ in $\overline N'$ or to $\nicefrac{{\I_{\overline N'\cap \widetilde N}(\la')}}{T'}$ because of the Steinberg relations.\end{proof}

\section{Application \texorpdfstring{to $\ovG^F=\tw 2 \tD_{l,\SC}(q)$}{to twisted groups of type D}} \label{appl2D}
In this section we roughly sketch the proof for the remaining case, where $\ovG$ has a root system of type $\tD_l$ and $F$ is the product of a standard Frobenius endomorphism and a graph automorphism of order $2$.

\begin{setting} In this section we use the notation for $\ovR$, $\overline \RF$, $ \ovovG$, $\ovovX_\al$, $R$, $\RF$ and $\ovG$ from \ref{set51D}. Let $\gamma$ be the automorphism of $R$, stabilising $\RF$ and interchanging $\al_1$ and $\al_2$, $\Gamma:\ovG\rightarrow \ovG$ the associated graph automorphism as in \ref{setgen} and $F=F_0\circ \Gamma$, where $F_0$ is the standard Frobenius endomorphism associated to $q$. Further let $\oF_0:\ovovG \rightarrow \ovovG$ be the analogous standard Frobenius endomorphism of $\ovovG$ and $\phi$ the automorphism of the cocharacter lattice $Y$ of $\ovG$, which is determined by $F$ as in \ref{setgen}. The associated automorphism $\overline \phi$ of $Y\otimes \CC$ satisfies $\phi(e_i)=\begin{cases}\,\,\,\,\, e_i&i\neq 1,\\
-e_1& i=1.\end{cases}$ Assume $d$ to be a regular number of $(\ovG,F)$, more precisely by Table \ref{TabregZahl} a divisor of $2(l-1)$ or a divisor of $2l$, such that $\frac {2l} d$ is odd. 
\end{setting}
We now construct $\oF: \ovovG \rightarrow \ovovG$ such that $\restr \oF\,|{\ovG}= F$.

\begin{lem}[{\cite[12.1.3]{Spaeth}}]\label{Defn}
Let $\zeta$ be a primitive $(2,q-1)^2$th root of unity in $\ovF_q$ and $\overline n_{i}$ the generators of the extended Weyl group $\overline V$ of $\overline \ovG$ as in \ref{EinfV}. The element $n:= \overline n_{1}\prod_{i=2}^{l} \widetilde h_{e_{i}}(\zeta) \in \overline \ovG$ satisfies $\calG(x)= x^n \text{ for all }x \in \ovG$ and $\rho(n)$ induces $\phi$ on $Y$.
\end{lem}
\begin{proof} Detailed calculations with the Steinberg relations imply $\calG(x)= x^n$ for all $x \in \ovX_\al$ ($\pm \al \in \RF$) and this implies the equation. The second part immediately follows from the first part.
\end{proof}

As in the previous section we obtain a Sylow $d$-twist of $(\ovG, F)$ by using one for a group of type $\tB_l$ or $\tB_{l-1}$.
\begin{lem}\label{defv2dl}
Let $l':=\begin{cases} l& d\teilt 2l \text{ and }2\nmid \frac {2l}{d},\\ l-1& \text{otherwise},\end{cases}$ $R'$ the root subsystem of $\overline R$ of type $\tB_{l'}$ with simple roots $\RF':=\Lset{\overline \al_1, \ldots \al_{l'}}$, $\ovG':=\Spann<\ovX_\al|\al \in \overline R'>\leq \overline \ovG$ the simply-connected simple algebraic group associated to $R'$, $v'$ a Sylow $d$-twist of $(\ovG', \restr \overline F\,|{\ovG'})$ and $ v:=\begin{cases}
v'n^{-1} & v' \notin \ovG,\\ 
v'n_{e_{l}}(1) n^{-1} &v' \in \ovG.\end{cases}$ Then $ v\Gamma$ is a Sylow $d$-twist of $(\ovG, F)$.
\end{lem}
\begin{proof} This can be verified with the arguments of Lemma \ref{defvD1}, where one uses Lemma \ref{Defn}.
\end{proof}

This Sylow twist is useful in verifying Theorem \ref{Theo1}.
\begin{lem}\label{tildeN2D}
Maximal extensibility holds with respect to $\ovT^{vF}\lhd \ovN^{vF}$.\end{lem}
\begin{proof} 
Considerations analogous to the proof of Lemma \ref{Theo1DI} show the statement in the case of $l'=l$. If $l'=l-1$ transferring the arguments used for showing Lemma \ref{Theo1DII} leads to a proof of the statement. In both cases we compare $\ovT^{vF}\lhd \ovN^{vF}$ with Sylow $d$-Levi subgroups and Sylow $d$-normaliser for a group with a root system of type $\tB_{l'}$, where maximal extensibility is known to hold by Lemma \ref{Theo1B}.
\end{proof}

As we have considered all possible root systems of $\ovG$, the above completes the proof of Theorem \ref{Theo1}. The proofs of this and the preceding section show the following statement about the relative inertia group of $\la \in \Irr(T)$.
\begin{lem} \label{lem_relID2D}
With the assumptions of this or the preceding section let $\la \in \Irr(T)$ and $\overline I_{\tB}$ be the relative inertia group of the associated character $\la'\in \Irr(T')$, where $T'$ is a Sylow $d$-Levi subgroup of a simply-connected simple group of type $\tB_{l'}$ and $N'$ the associated Sylow $d$-normaliser. The group $\overline N_{\tB}:=\nicefrac {N'}{T'}$ is isomorphic to $\overline K_1\wr \Sym_{j}$ and by Remark \ref{rem_relI} $\overline I_{\tB}$ is a subgroup of $\overline K_1\wr \Sym_{j}$ and has a normal subgroup isomorphic to $(A_1\times \cdots \cdots A_j) \rtimes U$ with index $1$ or $2$, where $U$ is a direct product of symmetric groups.
Let $\overline N_{\tD}$ be the normal subgroup of $\overline N_{\tB}$, whose elements $((\zeta_1, \ldots, \zeta_j),\sigma)\in \overline N_\tB$ ($\zeta_r\in \overline K_1$) satisfy $(\prod_{r}\zeta_r)^{\frac{\betrag{\overline K_1}}2} =1$ and which is isomorphic to a relative Weyl group of a Sylow $d$-torus for a group of type $\tD_{l'}$. 
Then the relative inertia subgroup of $\la$ in $N$ is isomorphic to $\overline I_{\tB}$ or $\overline I_{\tB}\cap \overline N_{\tD}$.\end{lem}
\begin{proof}
This can be deduced from the proofs of the Theorem \ref{Theo1} in this situation.
\end{proof}

\section{Applications to the McKay conjecture} \label{McKay}
In this section we prove Corollary \ref{Cor1}, more precisely we verify the McKay conjecture for the groups $\ovG^F$ from \ref{Theo1} and the primes $\ell$ fulfilling the given condition. We have to assume that $\ell$ is different from the defining characteristic of $\ovG^F$ and the multiplicative order $d$ of $q$ in $\nicefrac \ZZ {\ell \ZZ}$ is a regular number of $(\ovG,F)$.

The author has proven a similar statement already for exceptional groups of Lie type and similar ideas lead here to a verification of the McKay conjecture as well. The following lemma is the main tool for proving the McKay conjecture, and was already implicitly presented in the proof of Theorem C of \cite{Spaeth_Preprint}.

\begin{lem} \label{lem_mckay}Let $\calL\lhd \calN$ be finite groups such that $\calL $ is abelian, maximal extensibility holds with respect to $\calL\lhd \calN$ and every group occurring as relative inertia group of some $\la\in \Irr(\calL)$ satisfies the McKay conjecture. Then the McKay conjecture holds for $\calN$.\end{lem}
\begin{proof}
Let $P$ be a Sylow $\ell$-subgroup and $\la \in \Irr(\calL)$ with $P\leq \I_\calN(\la)$. Using the Clifford correspondence, the assumption about $\nicefrac {\I_\calN(\la)}\calL$ and the theorem of Gallagher \cite[6.17]{Isa} one obtains $\betrag{\IRRl(\calN|\la)} = \betrag{\IRRl(I_0|\la)}$, where $I_0=\calL \I_{\NNN_\calN(P)}(\la)$ and $\IRRl(\calN|\la):=\IRR(\calN|\la)\cap \Irrl(\calN)$.

This implies $\betrag{\Irrl(\calN)}=\betrag{\Irrl(\calL\NNN_\calN(P))}$. As the group $\calL\NNN_\calN(P))$ is $\ell$-solvable the statement follows with the result of \cite{Wo}.
\end{proof}
\begin{abschnitt}[Application of Lemma \ref{lem_mckay} to wreath products]
One can easily apply the above result to groups of the form $A\wr \Sym_l$, whenever $A$ is an abelian group. All groups occurring as relative inertia groups are direct products of symmetric groups. As the symmetric groups fulfil the McKay conjecture by \cite{Ol}, the McKay conjecture holds for the group $A\wr \Sym_l$, by Lemma \ref{lem_mckay}.
\end{abschnitt}

In order to apply the above lemma we have to verify that the groups occurring as relative inertia groups satisfy the McKay conjecture.
\begin{lem}
Assume that the McKay conjecture holds for $\ell$ and every finite group $\widetilde U$, which has a normal subgroup isomorphic to the direct product of symmetric groups, whose associated quotient is abelian.
Let $(\ovG, F)$ be like in Theorem \ref{Theo1}, $L$ an abelian Sylow Levi subgroup of $(\ovG,F)$, $N$ the associated Sylow normaliser and $\la\in \Irr(L)$. Then the McKay conjecture holds for $\I_N(\la)/L$ and every prime $\ell$.\end{lem}
\begin{proof} We prove this statement in several steps. First we determine which relative inertia groups $\overline I$ occur for the situation of Theorem \ref{Theo1}. In a second step we define a normal abelian subgroup $A\lhd \overline I$, such that a maximal extensibility property holds. By Lemma \ref{lem_mckay} it suffices to check that the relative inertia subgroups for $A\lhd \overline I$ satisfy the McKay conjecture.

Assume first $R$ to be a root system of type $\tA_{l-1}$, $\tB_l$ or $\tC_l$. As here the proof of Theorem \ref{Theo1} is based on the application of Proposition \ref{mFreg} the relative inertia groups are given by Remark \ref{rem_relI}: let $\la \in \Irr(T)$. Without loss of generality we may assume that $\overline I:=\nicefrac {\I_N(\la)}{T}$ has a normal subgroup $\overline I_0:=(A_1 \times\ldots \times A_j )\rtimes U$, where $A_i$ and $U$ are as in Remark \ref{rem_relI}. The index of $\overline I_0$ in $\overline I$ is a divisor of the order of $\delta$ and the quotient is cyclic.

If $R$ is of type $\tC_l$, the proof of Lemma \ref{Theo1C} was a direct consequence of Proposition \ref{mFreg} with a trivial character $\delta$. In the case where $R$ is a root system of type $\tA_l$, one observes from the proofs of the lemmas \ref{Adl}, \ref{Adl-1} and \ref{Theo12A} that the relative inertia groups are given above. If $R$ is of type $\tB_{l}$, the proof of Lemma \ref{Theo1B} the relative inertia group given above holds as well.

Let $\la \in \Irr(T)$ and $A=A_1 \times\ldots \times A_j $. Then maximal extensibility holds with respect to $A\lhd \overline I$. Let $\psi \in \Irr(A)$. Then $\psi$ can be extended to $\widetilde \psi$ on $\overline K:=\spann<\overline K_r>$, where $\overline K_r$ is defined as in Remark \ref{rem_relI}, such that $\restr \wpsi|{\overline K_r}=\restr \wpsi|{\overline K_{r'}}$, whenever $\restr \psi|{\overline K_r}=\restr \psi|{\overline K_{r'}}$.(By the equation we use that there exist canonical isomorphisms between the groups $\overline K_r$.) Now maximal extensibility holds with respect to $\overline K\lhd \overline K\rtimes \Sym_j$ and $\widetilde \psi$ extends to a group $\overline K\rtimes \Sym_{\psi}$, where $\Sym_\psi$ is the direct product of symmetric groups acting on the sets $\Lset{i| \restr \psi|{\overline K_i}=\kappa}$ ($\kappa \in \Irr(\overline K_1)$). The group $\I_{\overline I}(\psi)$ is by definition of $\wpsi$ a subgroup of $\overline K\rtimes \Sym_{\psi}$. This shows that $\psi$ extends maximally in $\overline I$ and the relative inertia subgroup of $\psi$ in $\overline I_0$ is a direct product of symmetric groups and $\I_{\overline I_0}(\psi)$ is a normal subgroup of $\I_{\overline I}(\psi)$, with cyclic quotient. But for such groups the McKay conjecture holds by assumption. Hence Lemma \ref{lem_mckay} implies the statement.

We are left with the case where $R$ is a root system of type $\tD_l$. The relative inertia subgroups are known from Lemma \ref{lem_relID2D}. Here the relative inertia subgroup is as above or $\overline I=\overline I_{\tB}\cap \overline N_{\tD}$ where $\overline I_{\tB}\leq \overline K_1\wr \Sym_{j}$ and $\overline N_{\tD}$ are defined as in Lemma \ref{lem_relID2D}. 

Let $(A_1\times \cdots \times  A_j) \rtimes U$ be the normal subgroup of $\overline I_{\tB}$ with the structure as above, whose index in $\overline I_{\tB}$ is $1$ or $2$. Further let $A=(A_1\times \cdots \times A_j)\cap \overline N_{\tD}$. If $A=A_{\tB}$ for $A_{\tB}:=A_1\times \cdots \times A_j$ then maximal extensibility holds with respect to $A\lhd \overline I$ according to the previous arguments. If $A\neq A_{\tB}$, every character $\psi\in \Irr(A)$ can be extended to a character $\widetilde \psi \in \Irr( A_{\tB})$. By the above there exists a maximal extension $\widehat \psi$ of $\widetilde \psi$ to $\overline K\rtimes U_{\widetilde \psi}$, where $U_{\widetilde \psi}$ is a direct product of symmetric groups and $\restr \widetilde \psi|{U_{\widetilde \psi}}$ is trivial. Detailed calculations show that $x\in \I_{\overline I}( \psi)$ normalises $\overline K\rtimes U_{\widetilde \psi}$, the character $\restr \widehat \psi|{(\overline K\rtimes U_{\widetilde \psi})\cap \overline I}$ is invariant under conjugation with $x$ and hence $\psi$ can be extended to its inertia group in $\overline I$, which has ${(\overline K\rtimes U_{\widetilde \psi})\cap \overline I}$ as a subgroup of index $1$ or $2$.

Some more detailed calculations show that the relative inertia group of $\psi$ in $\overline I$ has a normal subgroup of index $1$, $2$ or $4$, which is isomorphic to a direct product of symmetric groups.

As every character of $A$ extends to its inertia group and the relative inertia groups satisfy the McKay conjecture by the assumption, Lemma \ref{lem_mckay} implies the McKay conjecture for $\overline I$.
\end{proof}

We show the assumption of the above lemma.
\begin{lem}\label{lem_Navarro}
The McKay conjecture holds for $\ell$ and every finite group $C$, which has a normal subgroup $S$ isomorphic to the direct product of symmetric groups, whose associated quotient is abelian.
\end{lem}
\begin{proof}
Let $S=S_2\times S_3 \times \cdots$ with $S_i=\Sym_{i}^{a_i}$ for every integer $i$, $A_i=\Alt_{i}^{a_i}$ and $A= \prod_{i\geq 5} A_i \lhd S$ be the direct of the simple alternating groups $\Alt_i$ on $i$ points. 
By Theorem 3.1 of \cite{ManonLie} we know that the simple alternating groups are good in the sense of \cite{IsaMaNa}. As every group $A_i$ is the product of simple good groups we may apply Theorem 13.1 of \cite{IsaMaNa} in this situations. This implies $\betrag{\Irrl(C)}= \betrag{\Irrl(\NNN_C(P_i))}$ for every integer $i\geq 5 $ and a Sylow $\ell$-subgroup $P_i$ of $A_i$. As $A_{i'}$ ($i'\neq i$) is again a normal subgroup of $\NNN_A(P_i)$ we can apply Theorem 13.1 again. Inductively this gives us $\betrag{\Irrl(C)}=\betrag{\Irrl(\NNN_C(P_0))}$, where $P_0$ is the direct product of the groups $P_i$. Hence $P_0$ is a Sylow $\ell$-subgroup of $A$. The group $\NNN_C(P_0)$ is $\ell$-solvable and the result of Wolf in \cite{Wo} gives us $\betrag{\Irrl(\NNN_C(P_0)))}=\betrag{\Irrl(\NNN_C(P))}$, where $P$ is a Sylow $\ell$-subgroup of $C$ with $P_0\lhd P$. This proves the statement.
\end{proof}

We have proven the necessary tools to verify Corollary \ref{Cor1}.
\def\proofname{PROOF OF COROLLARY \ref{Cor1}.}
\begin{proof}
Using Theorem \ref{Theo1} Malle showed in \cite[Theorem 7.8]{Ma06} the equation $\betrag{\Irrl(G)}=\betrag{\Irrl(\NNN_G(\ovS))}$, where $G:= \ovG^F$ and $\ovS$ is a regular Sylow $d$-torus of $(\ovG, F)$ associated to $\ell$ and $G$. The two lemmas above imply $\betrag{\Irrl(\NNN_G(\ovS))} = \betrag{\Irrl(\NNN_G(P))} $, where $P$ is a Sylow $\ell$-subgroup of $\NNN_G(\ovS)$, which is also one of $G$ by \cite[Theorem 7.8 (a)]{Ma06}.
\end{proof}

The given arguments prove the McKay conjecture also for $\ell\in \Lset{2,3}$, whenever Theorem 7.8 from \cite{Ma06} holds and the associated Sylow torus is regular.
\begin{rem}[Generalisation of Corollary \ref{Cor1} with $\ell=2,3$]\label{Ausnahmen}
The statement of Corollary \ref{Cor1} holds, if $\ell=2$ or $\ell=3$ except in the following cases:
\begin{itemize}
 \item $G=\Sp_{2n}(q)$ with $q\equiv 3, 5 (\operatorname {mod} 8)$ and $\ell=2$,
 \item $G=\SL_3(q)$ with $q\equiv 2,5 (\operatorname {mod} 9)$ and $\ell=3$,
 \item $G= \SU_3(q)$ with $q\equiv 2,5 (\operatorname {mod} 9)$ and $\ell=3$.
\end{itemize}
\end{rem}

\end{document}